\documentclass[a4paper,12pt]{amsart}
\usepackage{amsfonts}
\usepackage{amsthm}
\usepackage{amssymb}
\usepackage{amsmath}
\usepackage{enumerate}
\usepackage{url}
\usepackage{epsfig}
\usepackage{subfigure}

\usepackage{a4wide}

\allowdisplaybreaks

\begin{document}
\title{on the order of magnitude of Sudler products}
\author{Christoph Aistleitner}
\address{Institute of Analysis and Number Theory, 
TU Graz, Steyrergasse 30, 
8010 Graz, Austria; 
E-mail:aistleitner@math.tugraz.at}
\author{Niclas Technau}
\address{School of Mathematical Sciences, 
Tel Aviv University, Tel Aviv, Israel;
E-mail address: niclast@mail.tau.ac.il}
\author{Agamemnon Zafeiropoulos}
\address{Department of Mathematical Sciences, 
Norwegian University of Science and Technology,
NO-7491 Trondheim, Norway; 
Email address: agamemnon.zafeiropoulos@ntnu.no}

\thanks{CA is supported by the Austrian Science Fund (FWF), projects F-5512, I-3466, I-4945, P-34763 and Y-901. \\
\indent NT is supported by the European Research Council (ERC) under the European Union's Horizon 2020 \\
\indent research and innovation programme 
(Grant agreement No. 786758). \\
\indent AZ is supported by a postdoctoral fellowship funded by Grant 275113 of the Research Council of Norway} 

\begin{abstract} Given an irrational number $\alpha\in(0,1)$, the Sudler product is defined by $P_N(\alpha) = \prod_{r=1}^{N}2|\sin\pi r\alpha|$. Answering a question of Grepstad, Kaltenb\"ock and Neum\"uller we prove an asymptotic formula for distorted Sudler products when $\alpha$ is the golden ratio $(\sqrt{5}+1)/2$ and establish that in this case $\limsup_{N \to \infty} P_N(\alpha)/N < \infty$.  We obtain similar results for quadratic irrationals $\alpha$ with continued fraction expansion $\alpha = [a,a,a,\dots]$ for some integer $a \geq 1$, and give a full characterization of the values of $a$ for which $\liminf_{N \to \infty} P_N(\alpha)>0$ and $\limsup_{N \to \infty}   P_N(\alpha) / N  < \infty$ hold, respectively. We establish that there is a (sharp) transition point at $a=6$, 
and resolve as a by-product a problem of 
the first named author, Larcher, Pillichshammer, Saad Eddin, and Tichy. 
\end{abstract}
\date{\today}

\maketitle

\newtheorem{prop}{Proposition}
\newtheorem{claim}{Claim}
\newtheorem{lemma}{Lemma}
\newtheorem{thm}{Theorem}
\newtheorem{thma}{Theorem A}
\newtheorem{defn}{Definition}
\newtheorem{conj}{Conjecture}
\newtheorem{cor}{Corollary}

\newtheorem{quest}{Question}

\theoremstyle{definition}
\newtheorem{exmp}{Example}

\theoremstyle{remark}
\newtheorem{rmk}{Remark}

\newcommand{\be}{\beta}
\newcommand {\ve} {\varepsilon}
\newcommand{\venj}{\varepsilon_{n}^{(j)}}

\section{Introduction and statement of results -- the case of the golden ratio}

\subsection{Definition and   overview}
Given a real number $\alpha$ and an integer $N\geq 1,$ the {\em Sudler product} at stage $N$ with parameter $\alpha$ is defined as
\begin{equation} \label{defn}
	P_N(\alpha) = \prod_{r=1}^{N}2|\sin \pi r\alpha| . \end{equation}
When $\alpha= m/n$ is a rational number we can see that $P_N(\alpha) = 0$ for all $N\geq n$, so in this case the asymptotic behavior of $P_N$ as $N \to \infty$ is trivial. However, we note that for $\alpha=1/n$ and $N = n-1$ we have the important trigonometric identity $\prod_{r=1}^{n-1} 2\sin(\pi r/n)= n$, which in terms of the notation we introduced above can be written as $P_{n-1}(1/n)=n$; this identity will play a role throughout the paper. Since the periodicity of the sine function implies that $P_{N}(\alpha) = P_{N}(\{\alpha\})$, where $\{\alpha\}$ is the fractional part of $\alpha$, in order to study the asymptotic behavior of products $P_{N}(\alpha)$ as $N \to \infty$ we can restrict ourselves to irrational numbers $0\leq \alpha < 1$.\\

Sudler products have been studied in various contexts, as they appear to have connections with many different areas of research. The first known appearance of products of the form 
\eqref{defn} is in a paper of Erd\H{o}s and Szekeres \cite{erdos}.  There it is proved that 
$$ \liminf_{N\to\infty}P_N(\alpha)=0  \text{ and } \limsup_{N \to \infty}P_N(\alpha)=\infty \hspace{2mm} \text{ for almost all } \alpha,  $$
and it is conjectured that
\begin{equation} \label{erdosconj}
 \liminf_{N\to\infty}P_N(\alpha) = 0 \hspace{3mm} \text{ for all $\alpha$.}
\end{equation}
In \cite{erdos} it is also claimed without proof that the limit $E := \lim_{N\to \infty}\|P_N\|_{\infty}^{1/N}$ exists and $1<E<2.$ (Here $\|\cdot \|_{\infty}$ denotes the supremum norm on the interval $[0,1]$.) A formal proof of this claim was given by Sudler \cite{sudler}, who also gave a precise formula for the limit $E$ and provided asymptotic estimates for the points $\alpha_N\in (0,1)$ for which $\|P_N\|_{\infty} = P_{N}(\alpha_N) $. \\

The size of the $L_{\infty}$- as well as the $L_p$-norms of Sudler products and of certain subproducts has been studied extensively, and for more results we refer to \cite{atkinson, bell, jordanbell, bourgain, freiman, kol1, kol2, wright}. In the present paper we focus our attention on results concerning the pointwise growth of Sudler products. Towards this direction, the asymptotic estimate $|\sin\pi x | \asymp \|x\|, \, x\to 0 ,$ shows the intimate connection of the size of $P_N(\alpha)$ with the Diophantine properties of $\alpha$. (We write $\|x\|$ for the distance of $x$ to its nearest integer).  \\  

Pointwise estimates for Sudler products play a key role in the proof of the Ten Martini Problem by Avila and Jitomirskaya \cite{aj1}. Pointwise lower bounds for Sudler products (in the case of general irrational $\alpha$) also play a crucial role in \cite{aj2} by Avila, Jitomirskaya and Marx. Furthermore, we want to point to very recent work by Bettin and Drappeau \cite{bd1,bd2,bd3}. Following work of Zagier \cite{zagier}, they study the order of magnitude of the Kashaev invariant of certain hyperbolic knots, which in terms of the present paper could be described as an average of Sudler products with fixed rational parameter. Some aspects of their work also play an explicit (continued fractions, Diophantine approximation) or implicit (cotangent sums) role in the present paper, but some aspects (modularity, reciprocity formulas) do not play a visible role in our paper at all. In all the papers mentioned in this paragraph, the approach taken is somewhat different from the one taken in the present paper, but it seems very interesting to compare all these approaches and to find a unified picture. \\

Although the exponential growth of $\|P_N\|_{\infty}$ proved in \cite{sudler} could lead one to believe that the sequence $( P_N(\alpha) )_{N=1}^{\infty}$ also exhibits the same behaviour for most values of $\alpha$ (from the metrical point of view), it has been shown that this is not the case. Lubinsky and Saff \cite{lubinsky2} proved that $\lim_{N\to\infty}P_N(\alpha)^{1/N}=1 $ for almost all $\alpha$. Subsequently, Lubinsky \cite{lubinsky} proved several results which explicitly exhibit the underlying relation between the asymptotic order of magnitude of $P_N(\alpha)$ and the Diophantine properties of $\alpha$ as encoded in its continued fraction expansion $\alpha=[a_1,a_2,a_3,\ldots]$.  To be more specific, it is proved in \cite{lubinsky}  that for any $\ve>0$ we have 
$$ \log P_N(\alpha) \, \ll \, \log N (\log\log N)^{1+\ve}\quad \text{ as }\, N\to \infty,  \qquad \text{ for almost all } \alpha $$
(in this statement and throughout our paper, ``$\ll$'' is the usual Vinogradov symbol). We refer the reader to \cite[Theorem $1.1$]{lubinsky} for the more precise metrical statement. This result gives an upper bound on the order of magnitude of $P_N(\alpha)$ for typical $\alpha$ which is almost polynomial. On the other hand, Lubinsky showed that 
\begin{equation} \label{limsuplowerbound}
  \limsup_{N \to \infty} \frac{\log P_N(\alpha)}{\log N} \geq 1  \hspace{4mm} \text{ for all irrational } \alpha,    
\end{equation}
which means heuristically that the higher order of magnitude of $P_N(\alpha)$ is at least linear. Since then it has been conjectured 
by the first named author, Larcher, 
Saad Eddin, and Tichy \cite{products}) 
that equality is true in \eqref{limsuplowerbound}, 
and additionally that the even stronger statement 
\begin{equation} \label{folkloreconj}
\limsup_{N \to \infty} \frac{P_{N}(\alpha)}{N} <\infty 
\end{equation}
holds for all irrational $\alpha$. Green \cite{green} points out in his mathematical review of \cite{products} that not even the existence of a particular $\alpha$  satisfying \eqref{folkloreconj} is known, and mentions the specific case of $\alpha = \sqrt{2}$ as an example. Later in the paper we show that \eqref{folkloreconj} is indeed true for $\alpha=\sqrt{2}$ but fails for many quadratic irrationals: the maximal order of magnitude of  $P_{N}(\alpha)$  depends sensitively on the size of the partial quotients of $\alpha$, see Theorem \ref{lineargrowththeorem} and Corollary \ref{height corr}. \\

Lubinsky also showed in \cite{lubinsky} that 
\begin{equation} \label{liminf}
\liminf_{N \to \infty}P_N(\alpha) = 0
\end{equation} 
for any irrational $\alpha$ with unbounded partial quotients in its continued fraction expansion, while for irrational numbers $\alpha$ with bounded partial quotients (such numbers are called {\em badly approximable} in the context of Diophantine approximation) he showed that
$$  \log P_N(\alpha) \ll \log N, \hspace{4mm} \text{ as } N\to \infty, $$
where the implicit constant depends on $\alpha$. Thus for badly approximable $\alpha$ the  product $P_N(\alpha)$ has polynomial upper order of magnitude.

In addition to all results stated explicitly in \cite{lubinsky}, there is a statement which is mentioned as a byproduct of the proof of  \eqref{liminf} for irrationals with unbounded partial quotients; it is mentioned that relation \eqref{liminf} is actually also true for $\alpha$ with bounded partial quotients, provided that the partial quotients are infinitely often large enough. For a more detailed explanation of this phenomenon, we refer to Theorem $5.1$ of the survey paper \cite{sigrid3}.\\
 
More recently, Mestel and Verschueren \cite{mv} examined the behaviour of the sequence of Sudler products evaluated on the golden ratio $\frac{\sqrt{5}+1}{2}$. By periodicity the Sudler product for the golden ratio is the same as the Sudler product for $\phi := \frac{\sqrt{5}-1}{2} = \frac{\sqrt{5}+1}{2} -1 \approx 0.618\dots$, which is the conjugate of the golden ratio. Throughout our paper, $\phi$ will always denote this number (and by a very slight abuse of notation, we will call $P_N(\phi)$ the Sudler product for the golden ratio). Mestel and Verschueren established the convergence of the subsequence $P_{F_n}(\phi)$ to some positive and finite limit, where $F_n$ is the $n$-th Fibonacci number. Note that the Fibonacci numbers are the denominators of continued fraction approximations to $\phi$, which emphasizes the connection between the Sudler product and Diophantine properties of the parameter. The methods developed in \cite{mv} form the basis for many subsequent works as well as for the analysis in the current paper, and we will return to them with more details later in the course of the proofs.\\ 

Expanding the techniques of \cite{mv}, Grepstad and Neum\"uller \cite{sigrid2} showed the convergence of specific subsequences of $P_N(\alpha)$ when $\alpha$ is a {\em quadratic irrational}, while Grepstad, Kaltenb\"ock and Neum\"uller \cite{gkn} proved that for the specific case of the golden ratio we actually have 
\begin{equation} \label{liminf_pos}
\liminf_{N \to \infty}P_N(\phi) > 0, 
\end{equation}
thus disproving the conjecture \eqref{erdosconj} of Erd\H{o}s and Szekeres in \cite{erdos}.  This last result  is particularly striking, since it shows that whether $\alpha$ satisfies \eqref{liminf} or not depends on the actual size of the partial quotients of $\alpha$, and not only on whether they are bounded or not. As previously mentioned, this dependence on the size of partial quotients is a phenomenon which we also encounter in one of the main results of this paper.

\begin{rmk}
While the present paper was under review, the results obtained in this paper were extended to the case of general quadratic irrationals in \cite{gnz}. The methods developed in the present paper also play a key role in \cite{ab2}, which gives upper bounds on $\log P_N (\alpha) / \log N$ in terms of the partial quotients of $\alpha$, as well as in the papers \cite{ab3,ab1}, which address the connection with quantum knot invariants and the work of Bettin--Drappeau and Zagier which was mentioned several paragraphs above. A kind of perturbed Sudler products, somewhat similar to those in the present paper, appear in a theoretical physics context in a recent paper of Koch \cite{koch}, where they describe generalized eigenfunctions of the self-dual Hof\-stadter Hamiltonian at energy zero; the methods used there seem to be very different from the ones used in the present paper, and it would be interesting to clarify the connections. Finally, there is a very recent paper of Hauke \cite{hauke}, who showed that $\liminf_{N \to \infty} P_N(\phi)$ actually is a minimum and equals $P_1(\phi)$, and that $P_N(\phi)$ is maximized at numbers $N$ of the form $F_n-1$ for some $n$. He obtained similar results for other quadratic irrationals with small partial quotients in their continued fraction expansion; see \cite{hauke} for details. \\
\end{rmk}

\subsection{Main results, Part 1: The case of the golden ratio $\phi$}  \label{golden_case}

As noted in the introduction, throughout this paper we write $\phi=\frac{\sqrt{5}-1}{2}$ for the conjugate of the golden ratio. The continued fraction expansion of this number is $[1,1,1,\dots]$. Throughout the paper $(F_n)_{n=0}^{\infty}$ denotes the Fibonacci sequence defined by $F_0=0, F_1=1,$ and $F_{n+1}= F_n + F_{n-1}$ for all $n=1,2,\ldots$ The sequence of Fibonacci numbers is closely associated with $\phi $, since each $F_n$ is a denominator of a convergent of $\phi$. \\

The precise result obtained by Mestel and Verschueren in \cite{mv} is the following.

\begin{thma}[Mestel, Verschueren \cite{mv}] \label{th_mv}For the sequence $P_{F_n}(\phi)$, there exists a constant $C_1>0$ such that \begin{equation} \label{c}
C_1 = \lim_{n\to\infty} P_{F_n}(\phi).
\end{equation}
Moreover, for the same constant $C_1$ we have 
\begin{equation}  \label{mv2}
\lim_{n\to \infty} \frac{P_{F_n-1}(\phi)}{F_{n}} = \frac{C_1\sqrt{5}}{2\pi} \cdot
\end{equation}  
\end{thma}

Regarding the value of the constant $C_1$,  Mestel and Verschueren in \cite[Theorem $1$]{mv} give the approximate value $C_1\approx 2.407 .$ However, this value seems to be purely based on experimental observation, and what they actually prove is only that $C_1$ exists and that $C_1>0$.\\

Refining the arguments of \cite{mv}, Grepstad, Kaltenb\"ock and Neum\"uller \cite{gkn} showed that $\liminf\limits_{N \to \infty}P_N(\phi)>0$, based on an analysis of perturbed Sudler products of the form \vspace{-2mm}
$$\prod_{r=1}^{N}2| \sin \pi (r \phi + \ve)|,\vspace{-2mm} $$
where $\ve$ are some specific small parameters coming from the so-called Zeckendorff representation of positive integers (see below). The significance of this perturbed version of the standard Sudler products defined in \eqref{defn} will become apparent in the heuristic analysis after the statement of the main results of this paper. \\

Following the approach of Grepstad, Kaltenb\"ock and Neum\"uller in \cite{gkn}, we define perturbed Sudler products of the form
\begin{equation} \label{sudlereps}
P_{F_n}(\phi,\ve) = \prod_{r=1}^{F_n} 2 \left| \sin \pi\left( r \phi + (-1)^{n+1} \frac{\ve}{F_n} \right) \right|.
\end{equation}
Thus we consider the perturbed Sudler products as functions of a real variable $\ve$, rather than as quantities that arise in an ad-hoc way from the Zeckendorff expansion of specific positive integers. We point out that in our text the definition of $P_{F_n}(\phi,\ve)$ is different from the one given in \cite{gkn}. The role of the factor $(-1)^{n+1}$ within the argument of the sine function will become clear during the proofs; this alternating factor reflects the fact that the error terms in successive continued fraction convergents also have alternating signs.\\

As our first main result, we establish the convergence of $P_{F_n}(\phi,\ve)$ as a sequence of functions in the variable $\ve$. Note that we already know that $\lim\limits_{n \to \infty} P_{F_n}(\phi,0) = C_1$, where $C_1>0$ is \vspace{-1.5mm} \newline the  constant from \eqref{c}.\\

For the sake of convenience, throughout the text we use the notation
\begin{equation} \label{ur}
u(r) = 2\sqrt{5}\left( r-\frac{1}{\sqrt{5}} \left(\{r\phi\}-\frac{1}{2} \right) \right) ,\hspace{5mm} r=1,2,\ldots  
\end{equation}
\, 

\begin{thm}[Existence of limit function for perturbed Sudler products] \label{th1}
For every $\ve \in \mathbb{R}$, the limit $ \lim\limits_{n \to \infty} P_{F_n}(\phi,\ve)$ exists and is equal to 
\begin{equation} \label{limit}
G(\ve) = K |\ve\sqrt{5}+1|\cdot \prod_{r=1}^{\infty} \left|1 - \frac{(2\ve\sqrt{5}+1)^2}{u(r)^2} \right|,
\end{equation}
where $K>0$ is some absolute constant and the sequence $u(r), r\geq 1,$ is as in \eqref{ur}. The convergence is uniform on any compact interval where $G$ is nonzero.\\
\end{thm}
 
The recursive structure of the Fibonacci sequence allows us to calculate the exact value of $G(\ve)$ at a certain value of $\ve$. As a consequence, we are able to determine the precise value of the constant $C_1$ in Theorem A $1$.\\

\begin{thm} \label{th2}
Let $ C_1 = \lim\limits_{n\to\infty} P_{F_n}(\phi) $ be the constant in \eqref{c}, and let the sequence $(u(r))_{r=1}^{\infty}$ be defined by \eqref{ur}. Then
$$C_1 =  (1+\phi) \cdot \prod_{r=1}^{\infty} \left( 1 - \frac{1}{u(r)^2}\right)\left( 1- \frac{(1+2\phi)^2}{u(r)^2}\right)^{-1} . $$
\end{thm}

\noindent Since $G(0)=C_1$, we can calculate the constant $K$ in Theorem \ref{th1} and obtain a completely explicit formula for the function $G$.

\begin{cor} \label{co_golden} The limiting function $G$ satisfies 
\begin{equation} \label{gformula}
G(\ve) = (1+\phi) |\ve\sqrt{5}+1|\cdot \prod_{r=1}^{\infty}\left| 1 - \frac{(2\ve\sqrt{5}+1)^2}{u(r)^2}\right|\left|1- \frac{(2\phi+1)^2}{u(r)^2}\right|^{-1} .
\end{equation}
\end{cor}

Some of the key properties of the function $G(\ve)$ are captured by the proposition below. 

\begin{prop}[Lower bound for limit function, and one special value] \label{prop1} 
We have 
\begin{equation}\label{eq: preim}
G \left(-\frac{\phi}{\sqrt{5}} \right) = 1.
\end{equation}
Furthermore, we have $G(\ve)>1.01$ for all $\ve$ in the range
\begin{equation}\label{eq: varepsilon range}
\varepsilon \in (-0.26, 0.58).
\end{equation}
The function $G$ is continuous on $\mathbb{R}$. Finally, $G$ is a $C^{\infty}$ function and it is strictly log-concave (i.e. the logarithm of $G$ is strictly concave) in any interval whose endpoints are two consecutive roots of $G$.
\end{prop}

    \begin{figure}[h] \label{fig1}
        \includegraphics[scale=0.8]{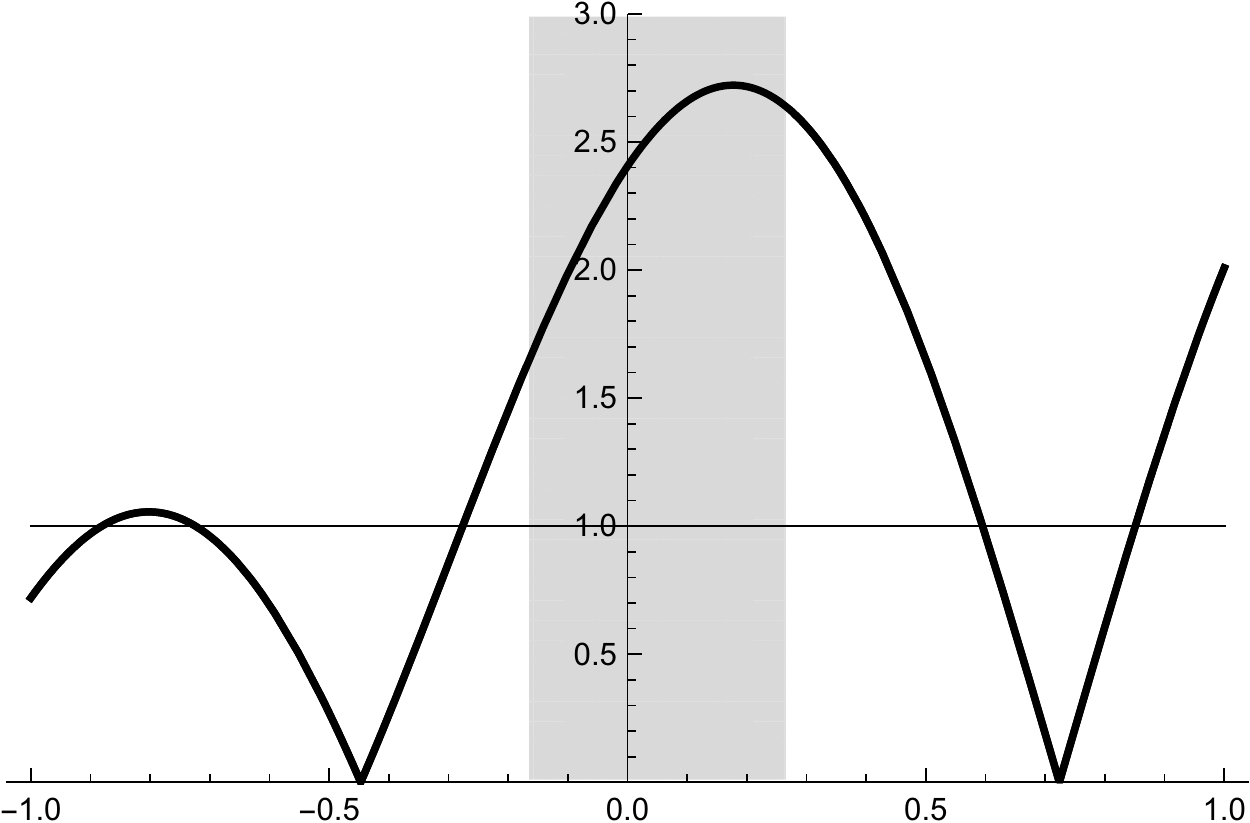}
        \caption{Plot of $G(\ve)$ in the range $-1 \leq \ve \leq 1$. The function has zeros at $\ve \approx -0.45$ and $\ve \approx 0.72$, and it equals 1 at $\ve \approx -0.28$ and $\ve \approx 0.60$. The fact that $\liminf P_N(\phi) > 0$, first proved in \cite{gkn}, can be reduced to the fact that $G(\ve)>1$ throughout the shaded range $\ve \in (-0.17, 0.27)$, since this turns out to be the range of possible perturbations $\ve$ coming from the Zeckendorff representation of positive integers. A very similar reasoning applies to the problem of establishing an upper bound for $\limsup P_N(\phi)/N$ by using a ``backward'' Zeckendorff expansion, see Theorem \ref{th3} below.}
    \end{figure}

To explain the significance of Theorems \ref{th1} and \ref{th2} and Proposition \ref{prop1}, we briefly indicate how (upper and lower) bounds for Sudler products can be calculated using the Zeckendorff representation of an integer. Let $N\geq 1$ be a positive integer. Then $N$ can be written in a unique manner in the form 
\begin{equation}\label{zeck}
N = F_{n_k} + F_{n_{k-1}} + \ldots + F_{n_1},
\end{equation}
where the integers $1\leq n_1 \leq \ldots \leq n_{k}$ are such that $n_{i+1} - n_i \geq 2$ for all $i=1,2,\ldots, k-1$. The representation of $N$ as in \eqref{zeck} is referred to as the Zeckendorff representation of $N$. We note that the Zeckendorff representation of integers is a special case of the Ostrowski expansion, to which we refer  later in the text. \\

In order to calculate the Sudler product $P_N(\phi)$, we expand $N$ into its Zeckendorff representation as above, split the full Sudler product into factors corresponding to the components of the Zeckendorff representation, and accordingly obtain
\begin{equation}  \label{zeckexpansion}
P_N(\phi)= \left( \prod_{r=1}^{F_{n_k}} 2|\sin\pi r\phi| \right) \left( \prod_{r=F_{n_k}+1}^{F_{n_k} + F_{n_{k-1}}} \!\!\!\!\!2|\sin\pi r\phi| \right) \dots \left(\prod_{r=F_{n_k}+\dots+F_{n_2} + 1}^{F_{n_k}+\dots+F_{n_1}}\!\!\!\!\!\!2|\sin\pi r\phi| \right).
\end{equation}
The first product on the right-hand side is equal to $P_{F_{n_k}}(\phi)$, and by Mestel and Verschueren's result it converges to $C_1$ as $N\to\infty$. The second product equals 
$$ \prod_{r=F_{n_k}+1}^{F_{n_k} + F_{n_k-1}}\hspace{-3mm} 2 |\sin \pi r \phi| = \prod_{r=1}^{F_{n_{k-1}}} 2 |\sin\pi (r \phi + F_{n_k} \phi)|, $$
so it is of the form  $P_{F_n}(\phi,\ve_k)$, 
as defined in \eqref{sudlereps},  with $\ve_k = (-1)^{1+n_{k-1}} F_{n_{k-1}}\{F_{n_k} \phi\}$. 
Consequently, by Theorem \ref{th1}, this factor is roughly $G(\ve_k)$, assuming that $n_{k-1}$ is ``large''.  Using the recursive structure of the Fibonacci sequence  --- see \eqref{phi1} and \eqref{phi2} below ---  we have $|\ve_k| =  F_{n_{k-1}}\phi^{n_k}$  and one can show that, for sufficiently large $N$, 
this implies that $\ve$ lies within the range  $(-\phi^3/ \sqrt{5}, \phi^2/\sqrt{5})$. The other products on the right-hand side of \eqref{zeckexpansion} are estimated in a similar way, and give contributions of size roughly $G(\ve_j)$ for some appropriate values of $\ve_j$ depending on the Zeckendorff expansion of $N$. Thus the problem of estimating the whole Sudler product $P_N(\phi)$ essentially boils down to estimating a product of values of the limit function $G$, evaluated at positions (``perturbations'') which depend on the Zeckendorff representation of $N$ in a relatively simple way.\\

Using continued fractions, it turns out that in such a product we can only encounter values of the perturbation variable which are within the range $(-\phi^2/\sqrt{5}, \phi/\sqrt{5})$, which is roughly $(-0.17, 0.27)$, with the possible exception of finitely many indices at the final segment of the representation, which correspond to sub-products of short length. Note that throughout the range $(-0.17, 0.27)$ we have $G(\ve)>1$ (cf.\ Proposition \ref{prop1} and Figure 1). Thus, heuristically speaking, $P_N(\phi)$ decomposes into a product of factors which are all at least $1$ (with the possible exception of finitely many indices at the final segment), and thus $\liminf\limits_{N\to\infty}P_N(\phi)$ cannot be equal to 0.\\

The analysis carried out in \cite{gkn} to show \eqref{liminf_pos} involves precisely these arguments, but the authors directly provide estimates for the perturbed products without having first established the convergence of the sequence $P_{F_n}(\phi,\ve)$ in \eqref{sudlereps}. A crucial advantage of the ``functional'' approach is that once we have established the existence of a limit function $G(\ve)$, together with a concavity property, it is not necessary anymore to estimate $G(\ve)$ for all perturbations $\ve$ that we might encounter during the proof, but it is rather sufficient that we show that the value of $G$ exceeds $1$ at two appropriate left and right endpoints; concavity of $\log G$ then tells us that  we must also have $G(\ve)>1$ everywhere in between.\\

In the sequel we  show that Theorem \ref{th1} also allows us to calculate \emph{upper} bounds for the size of the Sudler product: more precisely,  $P_N(\phi)$ grows at most linearly in $N$.

\begin{thm} \label{th3}
We have
$$ \limsup_{N \to \infty} \frac{P_N(\phi)}{N} < \infty. \vspace{4mm} $$
\end{thm}
 
\noindent  For the proof of Theorem \ref{th3} we will utilize the following decomposition. 
Given $N\geq 1$, when $n$ is such that $F_{n-1} \leq N + 1 < F_{n}$  we write 
\begin{equation} \label{reflectionprinciple}
 P_N(\phi) \,= \, \dfrac{ P_{F_{n}-1} (\phi)}{ \prod\limits_{r=N+1}^{F_{n}-1} 2 |\sin \pi r \phi|} \, = \, \frac{P_{F_{n}-1}(\phi) }{ \prod\limits_{r=1}^{F_{n}-N-1} 2 | \sin \pi (r-F_n) \phi | } \, \cdot
\end{equation}
The motivation for this decomposition is that by \eqref{mv2},  the quotients $P_{F_n-1}(\phi) /F_n $ are bounded from above, and since $F_{n-1}\leq N+1 < F_n$ we can show that $P_{F_n-1}(\phi)/N$ will be also bounded from above. It remains to control the product $\prod_{r=1}^{F_{n}-N-1} 2 | \sin \pi (r-F_n) \phi |$. This can be factorised as in \eqref{zeckexpansion}, but with an additional perturbation coming from the term $- F_n \phi$. Since $\|F_n\phi\|$ is sufficiently small, it will turn out that this additional perturbation does not cause any particular problems. \\

\noindent In other words, relation \eqref{reflectionprinciple} shows that the lower asymptotic order of magnitude of $P_N(\phi)$ reflects on the upper asymptotic order of magnitude of $P_N(\phi)/N$; for this reason, throughout the text we refer to \eqref{reflectionprinciple} as the reflection principle. \\

Thus the problem of finding a finite upper bound for $\limsup P_N(\phi )/N$ is, in a very natural sense, the dual problem of finding a positive lower bound for $\liminf P_N(\phi)$, and as sketched above this can be done by showing that $G(\ve)>1$ in an appropriate range of values of $\ve$. (In this particular case the range turns out to be roughly $(-0.10,0.17)$, which is even smaller than the range of possible perturbations in the $\liminf$ problem). See Section \ref{sec:proofs_phi} below for details.

\subsection{Main results, Part 2: The case of quadratic irrationals $\beta=[b,b,b,\ldots]$} \label{sec_main_2}
As already mentioned, it has been  conjectured that $\limsup_{N \to \infty} P_N(\alpha)/N$ is finite for \emph{all} irrationals $\alpha$. However, we will show that this is false even if $\alpha$ is restricted to the class of quadratic irrationals whose continued fraction expansion is of the simplest possible form. \\

\noindent For any positive integer $b\geq 1$, let 
$$ \beta = \beta(b) = [b,b,b,\ldots] $$
be the quadratic irrational with all partial quotients in its continued fraction expansion  equal to $b$. It is well-known that 
\begin{equation}\label{eq: explicit form of beta}
\beta=\frac{1}{2}(-b+\sqrt{b^2+4}),
\end{equation}
and if $(q_n)_{n=0}^{\infty}$ are the denominators of the convergents of $\beta$, then by induction we have that
\begin{equation}\label{b1}
\beta = \frac{q_{n-1}}{q_n} + (-1)^{n}\frac{\beta^{n+1}}{q_n}, \,\hspace{5mm} n=1,2,\ldots
\end{equation}
and
\begin{equation} \label{b2}
q_n = \frac{1}{\sqrt{b^2+4}} \left(\beta^{-(n+1)}-(-\beta)^{n+1} \right), \,\hspace{5mm} n=1,2,\ldots 
\end{equation}

The generalisation of the theorem of Mestel and Verschueren (Theorem A $1$ above) to the case of quadratic irrationals of the form $\be = \be(b)$ looks as follows.

\begin{thma} \label{generalthm}
	Let $\be=[b,b,b,\ldots]$, and let $(q_n)_{n=1}^{\infty}$ be the sequence of denominators associated with its continued fraction expansion. There exists a constant $C_b>0$ such that
\begin{equation} \label{cb}
C_b = \lim_{n\to\infty} P_{q_n}(\be).
\end{equation}
Moreover, for the constant $C_b$ we have
\begin{equation} \label{cb2}
\lim_{n\to\infty} \frac{P_{q_n-1}(\be)}{q_n} = \frac{C_b \sqrt{b^2+4}}{2\pi} \cdot
\end{equation}
\end{thma}

Relation \eqref{cb} is a special case of \cite[Theorem 1.2]{sigrid2} by Grepstad and Neum\"uller, which covers the case of arbitrary quadratic irrationals. The limit in \eqref{cb2} is not stated explicitly in \cite{sigrid2}, but it can be deduced easily from \eqref{cb} arguing as in \cite[Corollary 8.1]{mv}.\\

We prove an analogue of Theorem \ref{th1} for perturbed Sudler products in the case of the irrational $\be$. The objects of our study are now the functions 
\begin{equation}\label{pqndefinition}
P_{q_n}(\be,\ve) = \prod_{r=1}^{q_n}2\left|\sin \pi \left(r\be + (-1)^n \frac{\ve}{q_n} \right)  \right|.
\end{equation}
When the irrational $\beta$ is fixed, $q_n = q_n(\be)$ always denotes the denominator of the $n$-th convergent of $\be$. Furthermore, for convenience we write $P_{q_n}(\be,\ve)$ instead of $P_{q_n(\be)}(\be,\ve)$, which would be more accurate.
We also define 
\begin{equation} \label{ubr}
u_b(r)=2\sqrt{b^2+4}\left(r-\frac{1}{\sqrt{b^2+4}}\left(\{r\beta\}-\frac{1}{2}\right)\right), \hspace{2mm} r=1,2,\ldots
\end{equation}
Note that these formulas are straightforward generalisations of the ones for the case of the golden mean, which corresponds to the case $b=1$ and which we considered in Section \ref{golden_case} above.

\begin{rmk}The difference of the factors $(-1)^n$ and $(-1)^{n+1}$ in \eqref{sudlereps} and \eqref{pqndefinition} corresponds to the fact that the $n$--th convergent of $\phi$ has denominator $F_{n+1}$. We have adopted this notation in our results relevant to the golden ratio in order to remain consistent with \cite{mv}.
\end{rmk} 

The following result is a generalisation of Theorem \ref{th1}. \\

\begin{thm}[Existence of limit function for perturbed Sudler products] \label{th6}
	For every $\ve \in \mathbb{R}$, the limit $ \lim\limits_{n \to \infty} P_{q_n}(\be,\ve)$ exists and is equal to 
	\begin{equation} \label{limit_pert}
	G_\be(\ve) = K_b \cdot |\ve\sqrt{b^2+4}+1|\cdot \prod_{r=1}^{\infty} \left|1 - \frac{(2\ve\sqrt{b^2+4}+1)^2}{u_b(r)^2} \right|,
	\end{equation}
	where $K_b>0$ is some absolute constant and $u_b(r)$ is defined in \eqref{ubr}. The convergence is uniform on any compact interval where $G_{\beta}$ is nonzero. \vspace{1mm} \end{thm}

We can also prove generalised versions of Theorem \ref{th2} and of Corollary \ref{co_golden}. 

\begin{thm} \label{cbvaluethm}
Let $C_b>0$ be the constant as in \eqref{cb}. Then
\begin{equation} \label{cbvalue}
C_b^b = \frac{1}{\beta(\beta+1)\cdots(\beta+b-1)} \prod_{r=1}^{\infty}\prod_{j=1}^{b}\left(1- \frac{1}{u_b(r)^2}\right) \left( 1 - \frac{(2b-2j+2\beta+1)^2}{u_b(r)^2} \right)^{-1}.
\end{equation}
\end{thm}

\begin{cor} The limiting function $G_\be(\ve)$ in \eqref{limit_pert} satisfies
\begin{equation} \label{pinftyformula}
G_\be(\ve)^b = \frac{|\ve\sqrt{b^2+4}+1|^b }{ \beta \cdots(\beta+b-1)} \cdot \prod_{r=1}^{\infty}\prod_{j=1}^{b}\dfrac{ \left|1 - \dfrac{(2\ve\sqrt{b^2+4}+1)^2}{u_b(r)^2} \right|}{ \left|1 - \dfrac{(2b-2j+2\beta+1)^2}{u_b(r)^2} \right|} \cdot
\end{equation}	
\end{cor}

The following proposition states the basic properties of the function $G_{\be}$ that we use later in the paper. We omit the proof since it involves precisely the same convergence arguments as the proof of Proposition \ref{prop1}.

\begin{prop} \label{gbetaproperties}
For any $b\geq 2$, the function $G_{\be}$ defined in Theorem \ref{th6} is continuous on $\mathbb{R}$. Furthermore, it is a $C^{\infty}$-function and strictly log-concave in any interval with endpoints two of its consecutive roots. \vspace{1mm}
\end{prop}

Interestingly, the $\liminf$ result of Grepstad, Kaltenb\"ock and Neum\"uller and the $\limsup$ result of Theorem \ref{th3} cannot be extended to the general case of $b \geq 2$. Instead, it turns out that both results fail when $b$ is sufficiently large, and the proofs rely on the particular structure of the function $G_\be$ in an extremely delicate way. Theorem \ref{lineargrowththeorem} gives a full characterisation for this problem.\\

\begin{thm} \label{lineargrowththeorem}
Let $\be = [b,b,b,\ldots]$, where $b$ is a positive integer. Then the following holds.
\begin{itemize}
\item[(i)]  If $b \leq 5$, then $\liminf \limits_{N \to \infty} P_N(\be) > 0$ and $\limsup \limits_{N\to\infty}\dfrac{P_N(\be)}{N} < \infty.$ \vspace{1mm}
\item[(ii)] If $b \geq 6$, then $\liminf \limits_{N \to \infty} P_N(\be) = 0$ and $\limsup \limits_{N\to\infty}\dfrac{P_N(\be)}{N} = \infty.$
\end{itemize}
\end{thm}

The above theorem implies that \eqref{folkloreconj}
is not true for every irrational $\alpha$. In fact, 
we have a quantitative lower bound on 
the number of quadratic irrationals whose Sudler products 
have super-linear growth in $N$; here we are ordering 
the quadratic irrationals in a natural way --- by their naive height. 
For not interrupting the flow of presentation, 
we have recorded the relevant corollary, that is 
Corollary \ref{height corr}, right before the bibliography.\\

Sudler products for quadratic irrationals $\be=[b,b,b,\ldots ]$ were examined in \cite{sigrid3}, where it was proved that there exists some (finite) value $B_0$ such that $\liminf_{N \to \infty}P_N(\be)=0$ provided that $b \geq B_0$. The authors of \cite{sigrid3} showed that one can take $B_0 = e^{803}$, and, based on numerical calculations, they conjectured that actually $b=6$ should be the transition point where the $\liminf$ behaviour of $P_N(\be)$ changes. This conjecture is established in our theorem above. We note that the theorems cited and proved in this paper imply that
$$
\liminf_{N \to \infty} \frac{\log P_N(\be)}{\log N} = 0, \qquad \limsup_{N \to \infty} \frac{\log P_N(\be)}{\log N} = 1 \qquad \text{when \, $b \leq 5$},$$
and that modifying  the proof of Theorem \ref{lineargrowththeorem} we could also establish the slightly stronger conclusion that
$$ \liminf_{N \to \infty} \frac{\log P_N(\be)}{\log N} < 0, \qquad \limsup_{N \to \infty} \frac{\log P_N(\be)}{\log N} > 1 \qquad \text{when\, $b \geq 6$,} $$
instead of part (ii) of Theorem \ref{lineargrowththeorem}.\\

It turns out that for all $b \geq 7$ we have $C_b = G_\be(0) <1$ (see Corollary \ref{cor: when M-V const greater one} below). By continuity of $G_\beta$ this implies that $G_\beta(\ve)<1$ for all sufficiently small perturbations $\ve$, and we will see that this implies that $\liminf P_N(\be)=0$ as well as $\limsup P_N(\be)/N = \infty$ by a rather straightforward argument (see Lemma \ref{lemma>1} below). The case $b=6$ is special, since this is the only case where we have $C_b > 1$ but where 
$$
\liminf \limits_{N \to \infty} P_N(\be) = 0\qquad \text{and} \qquad \limsup \limits_{N\to\infty}\dfrac{P_N(\be)}{N} = \infty.
$$
We will need a separate proof for this case (Lemma \ref{lemma_b6} below). The case $b=1$ is the golden ratio case, for which the desired conclusions have been already established in Section \ref{golden_case} above. It remains to deal with the case when $b \in \{2, 3, 4, 5\}$. The analysis in these cases  is quite involved, and significantly more complex than in the case $b=1$. Roughly speaking, in the golden ratio case we could show that from the Ostrowski (Zeckendorff) representation there can arise no perturbations $\ve$ which lead to values of $G$ being smaller than $1$. This continues to be true when $b=2$, but in the case $3 \leq b \leq 5$ it can indeed happen that the Ostrowski representation leads to particular (negative) perturbations $\ve$ for which $G_\beta(\ve)<1$. However, we show that such problematic perturbations can only arise from very particular configurations of the Ostrowski representation, in such a way that contributions from an earlier stage compensate for the small factors, and the overall product still exceeds $1$. For a heuristic explanation of these effects see Figures $2-4$ below. 

\begin{figure}[h!]%
\centering
\subfigure[][]{%
\label{fig:ex3-a}%
\includegraphics[scale=0.5]{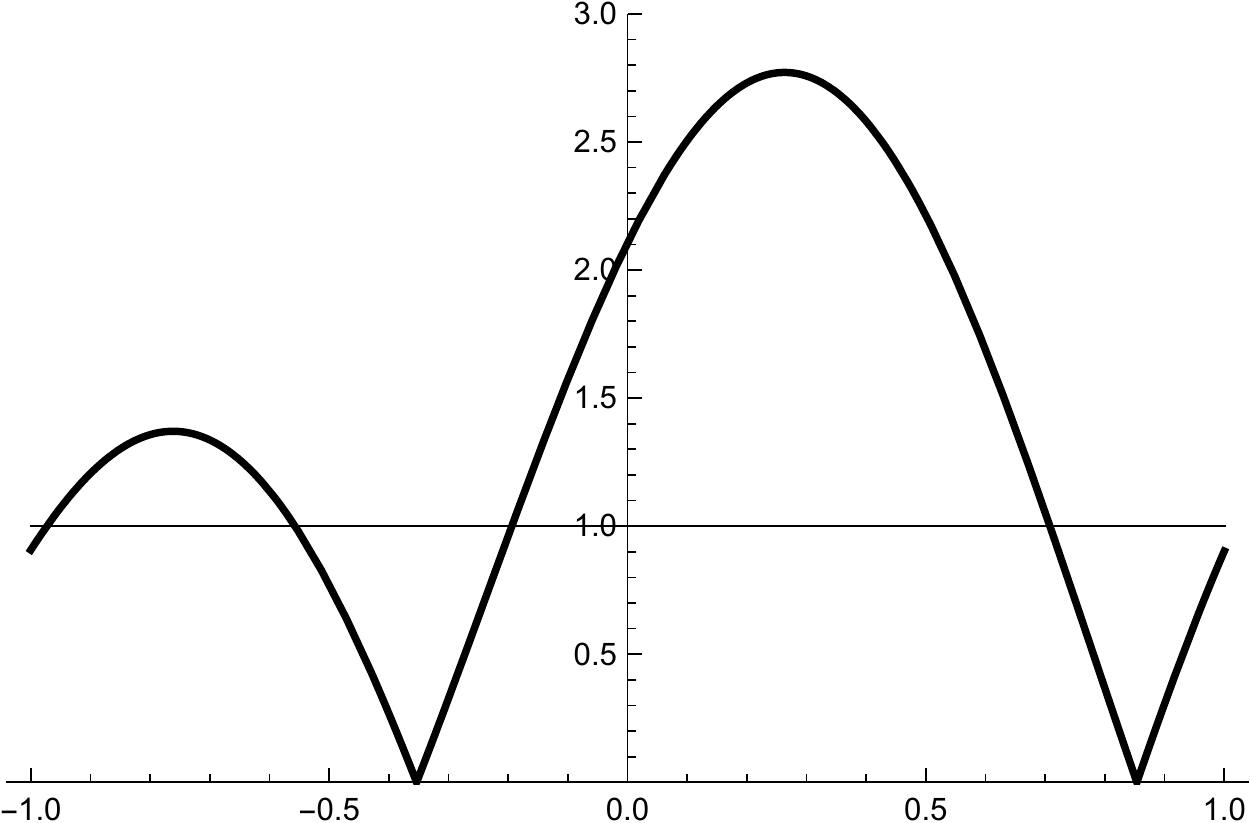}}%
\hspace{8pt}%
\subfigure[][]{%
\label{fig:ex3-b}%
\includegraphics[scale=0.5]{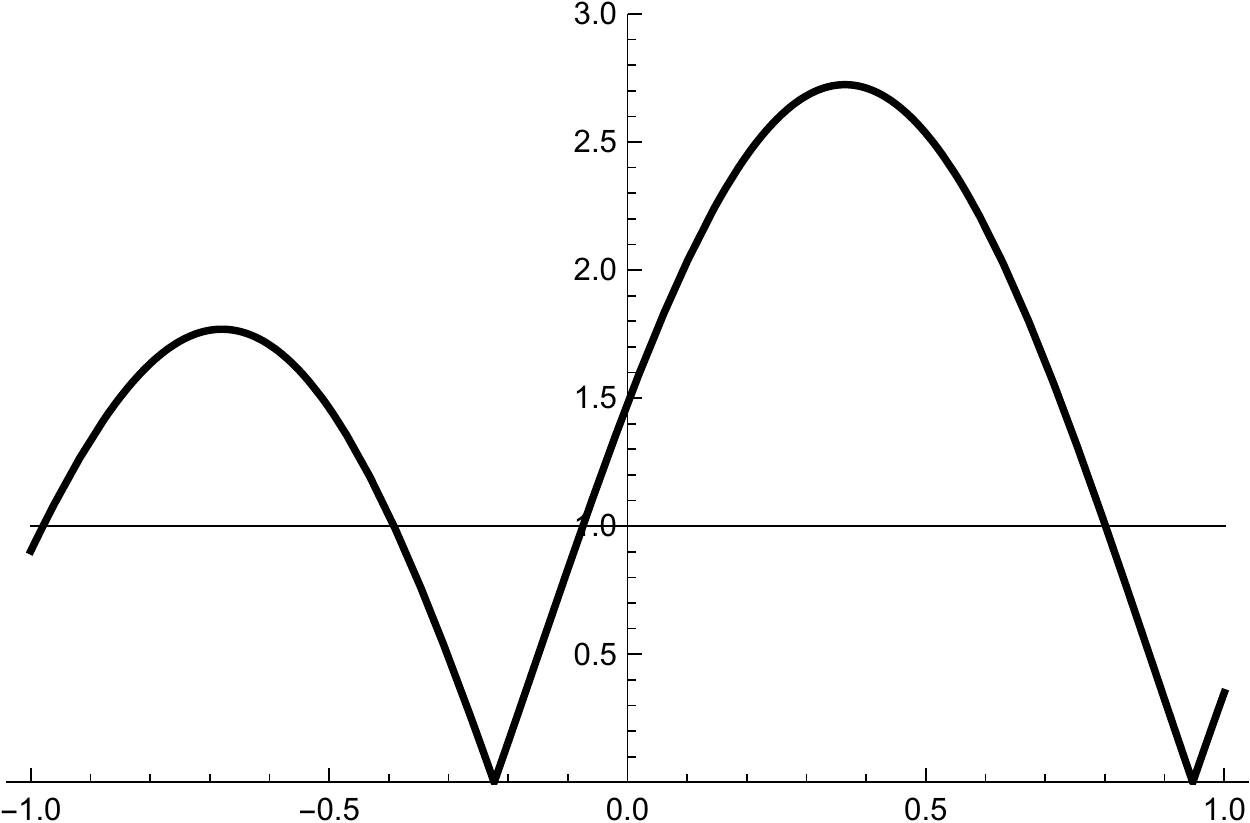}} \\
\subfigure[][]{%
\label{fig:ex3-c}%
\includegraphics[scale=0.5]{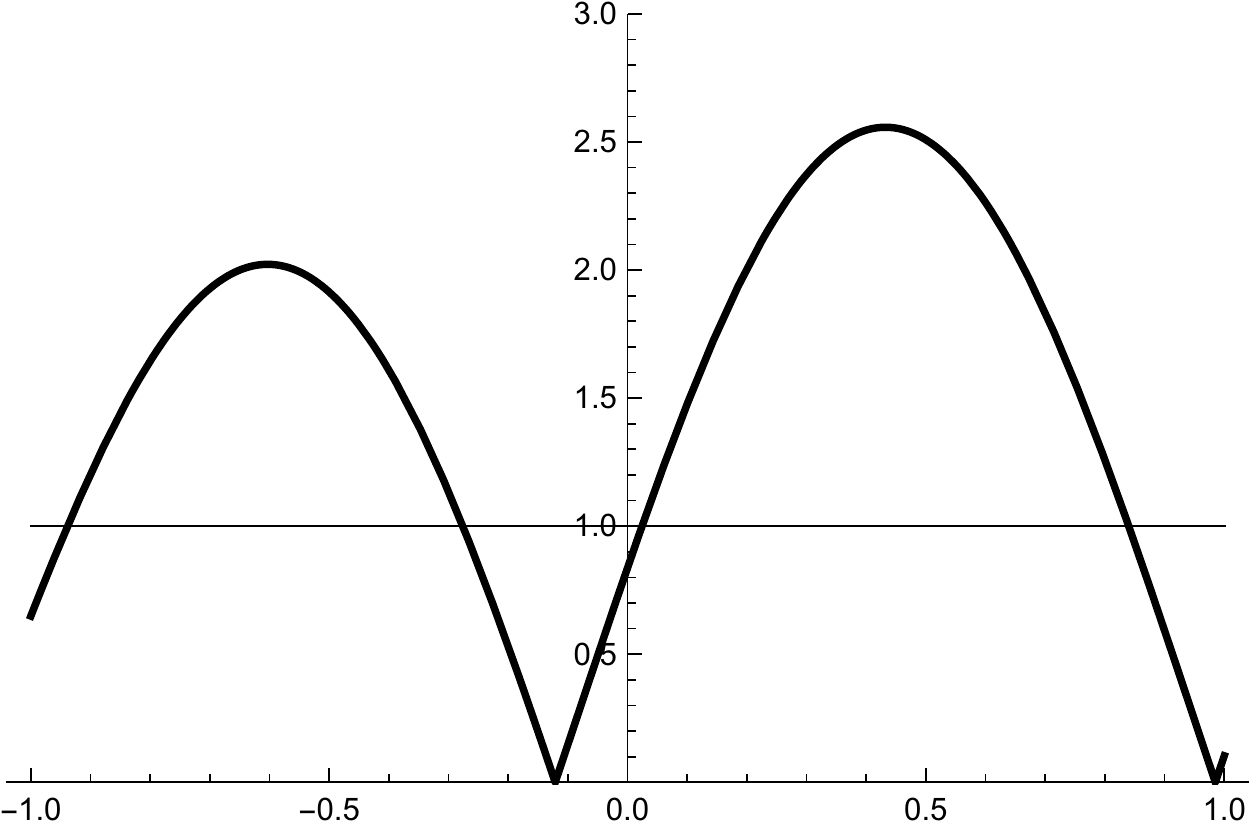}}%
\hspace{8pt}%
\subfigure[][]{%
\label{fig:ex3-d}%
\includegraphics[scale=0.5]{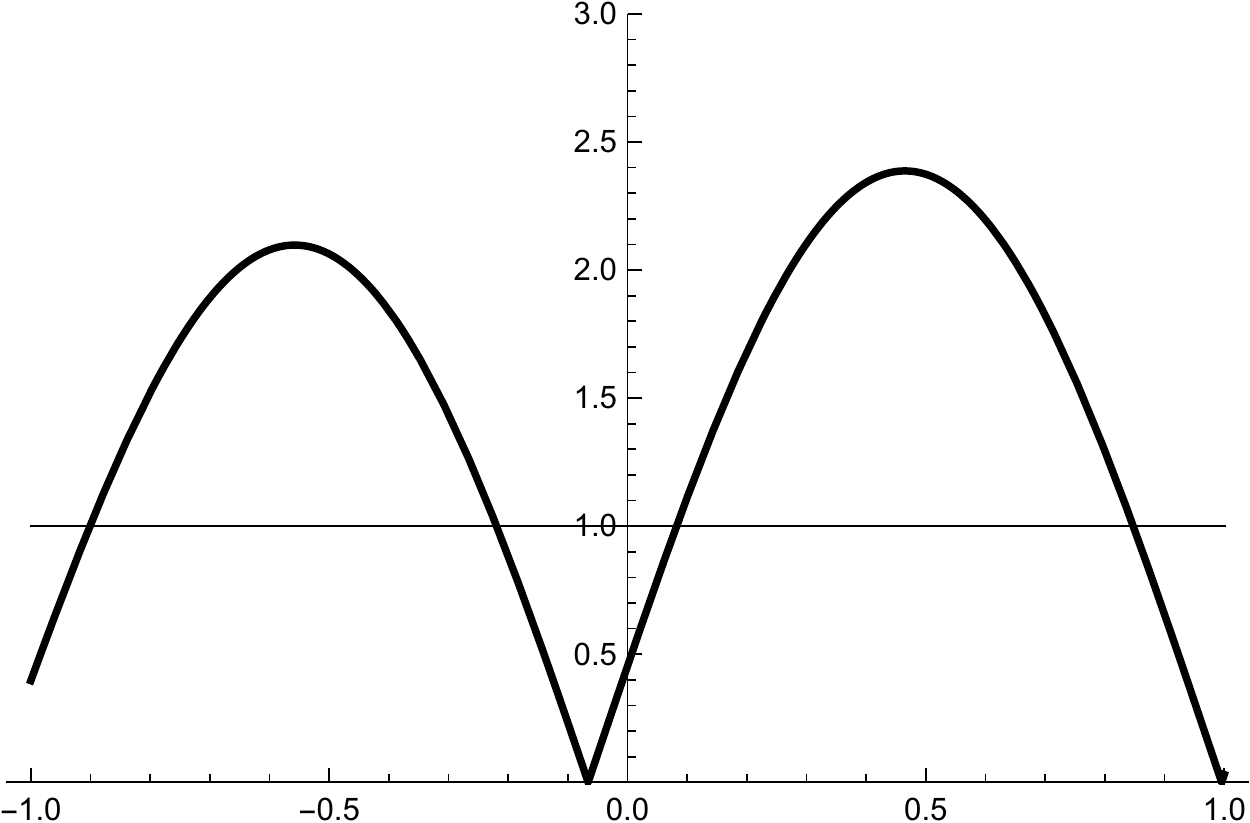}}%
\caption[A set of four subfigures.]{Plots of the limit functions $G_\beta(\ve)$ for different values of $\beta=[b,b,b,\ldots]$. In \subref{fig:ex3-a} we have $b=2$, in 
\subref{fig:ex3-b} we have $b=4$, in
\subref{fig:ex3-c} we have $b=8$, and in 
\subref{fig:ex3-d} we have $b=15$. Note how the peak of the function ``moves off'' to the right, so that for sufficiently large $b$ we have $C_b = G_\beta(0)<1$, which forces $\liminf P_N(\beta)$ to be 0 and $\limsup P_N(\beta)/N$ to be infinite (see Lemma \ref{lemma>1} below). In contrast, having $C_b = G_\be(0)>1$ is not sufficient to deduce that $\liminf P_N(\beta)=0$ and $\limsup P_N(\beta)/N<\infty$; in such a case, which can only happen when $b$ is small, detailed information on the size of $G_\beta$ throughout a certain range of possible perturbations $\ve$ is necessary.}%
\label{fig:ex3}%
\end{figure}

\begin{figure}[h!] \label{fig3}
        \includegraphics[scale=0.6]{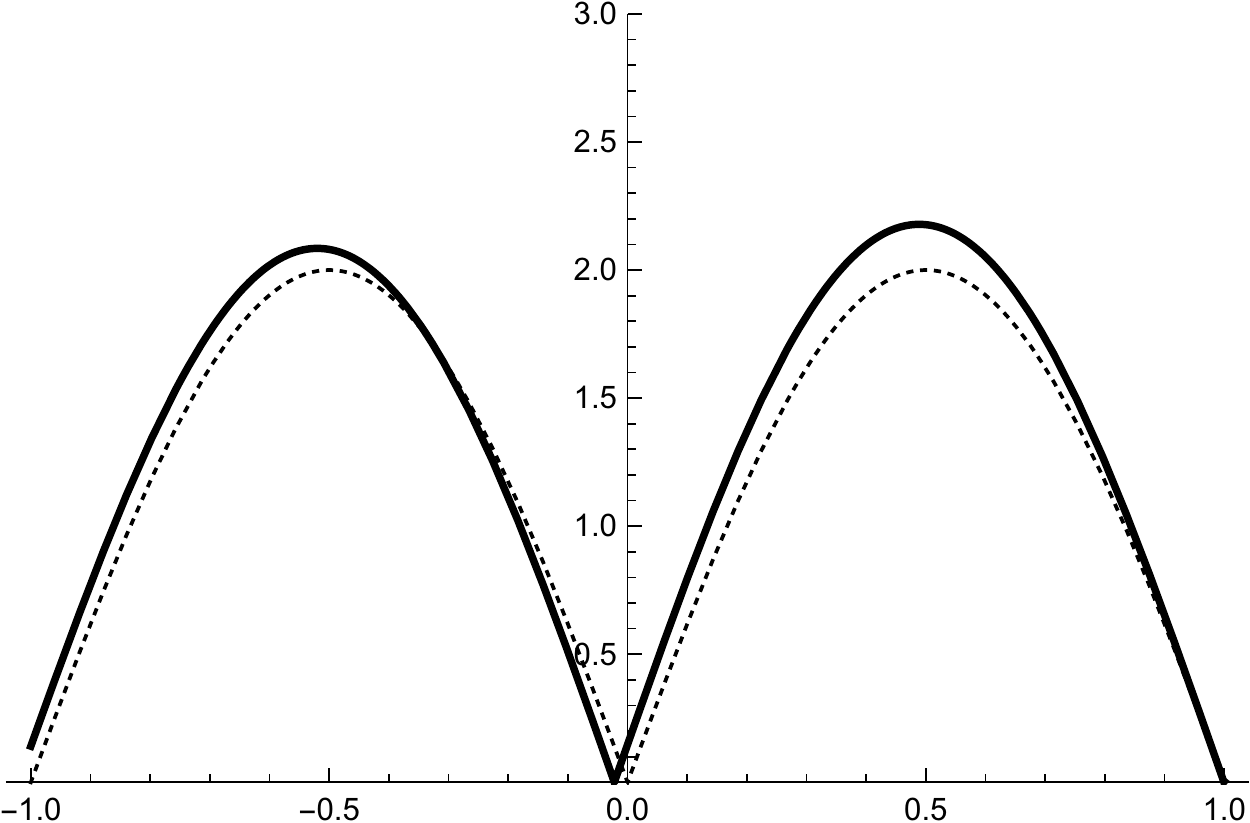}
        \caption{Plot of $G_\beta(\ve)$ for $b=35$ (solid line). Compare this with the plots in Figure 2. The peak has ``wandered off'' so far to the right that $G_\be$ only takes a very small value at $\ve=0$. The graph of $G_\be(\ve)$ has lost much of its similarity with the corresponding graphs for small values of $b$, and now rather resembles $2 |\sin \pi \ve|$ (dashed line).} 
    \end{figure}
    
\section{Proofs of the theorems for the golden ratio} \label{sec:proofs_phi}

Throughout this section, we will make use of the following relations:
\begin{equation}\label{phi1}
\phi = \frac{F_{n-1}}{F_n} + (-1)^{n+1}\frac{\phi^n}{F_n}, \,\hspace{5mm} n=1,2,\ldots
\end{equation}
and
\begin{equation} \label{phi2}
F_n = \frac{1}{\sqrt{5}} \left(\phi^{-n}-(-\phi)^n \right), \,\hspace{5mm} n=1,2,\ldots 
\end{equation}

\begin{proof}[Proof of Theorem \ref{th1}]
In order to show the convergence result of Theorem \ref{th1}, we follow the same factorisation argument as in \cite{mv} and \cite{gkn}. Here we present the basic steps of the proof and refer to \cite{gkn, mv} for the remaining details which are nearly identical. \\

In what follows, we assume that an integer $n\geq 1$ is given, and we write $[k]$ for the residue of the integer $k\in\mathbb{N}$ modulo $F_n$. Following \cite{mv} we define 
\begin{eqnarray*}
	s_n(0,\ve) &=& 2\sin\pi\left(\frac{\ve}{F_n}+\frac{\phi^{n}}{2} \right), \\
	s_n(r) &=&  2\sin\pi\left( \frac{r}{F_n}-\phi^{n}\left(\frac{[F_{n-1}r]}{F_n} -\frac{1}{2}\right)\right), \qquad r=1,2,\ldots, F_n-1.
\end{eqnarray*}
The product $P_{F_n}(\phi,\ve)$ may be further factorized as
\begin{equation} \label{factors} 
P_{F_n}(\phi,\ve) = A_n(\phi,\ve)\cdot B_n(\phi)\cdot C_n(\phi,\ve), \end{equation}
where   
\begin{eqnarray*}
	A_n(\phi,\ve) &=& 2 F_n\left| \sin \pi\left(F_n\phi + (-1)^{n+1}\frac{\ve}{F_n} \right)\right|,\\
	B_n(\phi) &=& \prod_{r=1}^{F_n-1} \frac{s_n(r)}{2\left| \sin(\pi r/F_n)\right|}, \hspace{3mm} \text{ and} \\
	C_n(\phi, \ve) &=& \prod_{r=1}^{F_n-1}\left( 1-\frac{s_n(0,\ve)^2}{s_n(r)^2} \right)^{\frac{1}{2}} .
\end{eqnarray*}
The proof of \eqref{factors} is  based on elementary trigonometric identities, see \cite[Lemma 5.1]{mv} or \cite[Lemma 3.1]{gkn}. Using properties \eqref{phi1} and \eqref{phi2} together with the asymptotic estimate $\sin x \sim x, \, x\to 0,$ we deduce that
\begin{eqnarray*}
A_n(\phi,\ve) & = & 2F_n \left|\sin\pi \left((-1)^{n+1}\phi^{n} +(-1)^{n+1}\frac{\ve}{F_n} \right) \right| \\
& = & 2F_n\left|\sin\pi\left( \phi^{n} + \frac{\ve}{F_n}\right)\right| \\
& \sim & 2\pi |F_n\phi^n +\ve| \\
&\sim & \frac{2\pi}{\sqrt{5}}|\ve\sqrt{5}+1|, \qquad n\to \infty .
\end{eqnarray*}

Regarding the products $B_n(\phi)$, it is shown in \cite{mv} that there exists a constant $B>0$ such that $B_n(\phi)\to B$ as $n\to \infty$. This is the most difficult part of the proof of Theorem A $1$ in \cite{mv}, but since the factor $B_n(\phi)$ does not depend on the perturbation $\ve$ at all we can just use this fact without any further work.\\

Finally, regarding the factor $C_n(\phi,\ve)$ we can show, arguing as in \cite[Section 6]{mv}, that 
$$ \lim_{n\to\infty}C_n(\phi,\ve)^2 = \prod_{r=1}^{\infty}\left(1-\frac{(2\ve\sqrt{5}+1)^2}{u(r)^2} \right)^2,  \hspace{3mm} \text{ for all } \ve\in\mathbb{R}, $$
where $u(r)$ is as in \eqref{ur}. The only difference in comparison with \cite{mv} is that the numerators $s_n(0,\ve)$ now depend on the perturbation $\ve$ and satisfy 
\begin{equation}\label{snestimate} 
s_n(0,\ve) = \pi\phi^n (2\ve\sqrt{5}+1)  + O(\phi^{3n})  \quad \text{ as } n\to \infty.
\end{equation}
Combining all the previous formulas, we obtain the requested convergence result for $P_{F_n}(\phi,\ve)$. \\

Now let $I$ be a compact interval.  In order to show that the convergence of $P_{F_n}(\phi,\ve)$ is uniform on $I$, it suffices to show that all three factors appearing in \eqref{factors} converge uniformly. This is trivial for $B_n(\phi)$, while for $A_n(\phi,\ve)$ it can be done using the estimate $\sin x = x + O(x^3),\, x\to 0$.  Finally regarding $C_n(\phi,\ve)$, the estimates \eqref{snestimate} as well as $(6.2)$ of \cite{mv} hold uniformly on $I$, and this allows us to deduce that $C_n(\phi,\ve)$ is uniformly Cauchy on $I$;  the details are left to the interested reader. 
 
\end{proof}

\begin{proof}[Proof of Theorem \ref{th2}]
Since by the Mestel--Vershueren Theorem the sequence $(P_{F_n}(\phi))_{n=1}^{\infty}$ has a limit $0<C_1 <\infty$, we get
\begin{eqnarray*}
1 &=& \lim_{n\to\infty} \frac{P_{F_{n+1}}(\phi)}{P_{F_{n-1}}(\phi)} \\
& = & \lim_{n\to\infty} \prod_{r=F_{n-1}+1}^{F_{n+1}}\!\!\!2|\sin \pi r\phi | \\
&=& \lim_{n\to\infty}\prod_{r=1}^{F_{n+1}-F_{n-1}}\hspace{-3mm}2| \sin\pi(F_{n-1}+r)\phi| \\
& = & \lim_{n\to\infty}\prod_{r=1}^{F_{n}}2| \sin\pi(r\phi + F_{n-1}\phi)| \\
& \stackrel{\eqref{phi1}}{=}& \lim_{n\to\infty}\prod_{r=1}^{F_{n}}2| \sin\pi \left(r\phi + (-1)^{n}\phi^{n-1} \right)| \\
& = & \lim_{n\to\infty}P_{F_n}(\phi, -F_n\phi^{n-1}) .
\end{eqnarray*}
Here \eqref{phi2} implies that 
\begin{equation} \label{asymptotic}
\ve_n := -F_{n}\phi^{n-1}  \sim -\frac{1}{\phi \sqrt{5}}, \qquad n\to \infty.
\end{equation}
At this point we can write 
\begin{equation} \label{factorisation}
P_{F_n}(\phi, \ve_n) = A_n(\phi,\ve_n)\cdot B_n(\phi) \cdot C_n(\phi, \ve_n),
\end{equation}  
where the factors appearing are as in \eqref{factors}. Now 
\begin{eqnarray*}
A_n(\phi, \ve_n) &=& 2F_n \left| \sin \pi\left(\phi^{n} -\phi^{n-1}\right)\right| \\
&\sim & 2\pi F_n \phi^{n}|1-\phi^{-1}| \\
&\sim & \frac{2\pi\phi }{\sqrt{5}} , \qquad n\to \infty.
\end{eqnarray*}
Regarding $B_n(\phi)$, we have already mentioned in the proof of Theorem \ref{th1} that there exists $B>0$ such that $$ \lim_{n\to\infty} B_n(\phi) = B.$$ Finally, for $C_n(\phi, \ve_n)$ one can employ the arguments of \cite[p.10--12]{mv} to prove that 
$$ \lim_{n\to\infty} C_n(\phi, \ve_n) = \prod_{r=1}^{\infty} \left( 1 - \frac{(1+2\phi)^2}{u(r)^2} \right).$$
Thus it follows by taking limits in \eqref{factorisation} that 
\begin{equation} \label{onelimit}
\frac{2\pi\phi}{\sqrt{5}} \cdot B \cdot \prod_{r=1}^{\infty} \left( 1 - \frac{(1+2\phi)^2}{u(r)^2} \right) = 1 .
\end{equation}
Regarding the constant $C_1>0$, in the proof of \cite[Theorem 3.1]{mv} it is actually shown that 
\begin{equation} \label{constantlimit}
C_1 = \frac{2\pi}{\sqrt{5}}\cdot B \cdot \prod_{r=1}^{\infty} \left( 1 - \frac{1}{u(r)^2} \right)
\end{equation}
Hence combining \eqref{onelimit} with \eqref{constantlimit} we obtain the requested relation.
\end{proof}

\begin{proof}[Proof of Theorem \ref{th3}]
To deduce Theorem \ref{th3} from Theorem \ref{th1}, let $N$ be given, and let $n=n(N)\geq 1$ be such that $F_{n-1} \leq N  + 1 < F_n$. Factorizing $P_N(\phi)$ as in \eqref{reflectionprinciple}, we obtain 
\begin{eqnarray}
	P_N(\phi) & = & \frac{P_{F_{n}-1}(\phi)}{\prod\limits_{r=N+1}^{F_{n}-1} 2 | \sin \pi r \phi |}   \, = \,   \frac{P_{F_n-1}(\phi)}{\prod\limits_{r=1}^{F_n-N-1} 2 | \sin \pi  (F_n - r) \phi |}  \nonumber\\ 
	& \stackrel{\eqref{phi1}}{=} &  \frac{P_{F_{n}-1}(\phi)}{\prod\limits_{r=1}^{F_n-N-1} 2|\sin\pi ( r \phi +(-\phi)^{n})|} \cdot \label{line_y}
\end{eqnarray}
Now let $$ F_n - (N+1) = F_{n_k} + F_{n_{k-1}} + \ldots + F_{n_1}$$ be the Zeckendorff representation of $F_n-(N+1)$ as in \eqref{zeck}. Since $F_{n-1} \leq N +1 < F_n,$ we have $n_k \leq n-2$. Set $N_k=0$ and $N_i = F_{n_k} + \ldots + F_{n_{i+1}}$ for $i=1,2,\ldots, k-1$. The product in the denominator in line \eqref{line_y} can then be written as 
\begin{eqnarray}
	\prod\limits_{r=1}^{F_n-N-1} 2|\sin\pi ( r \phi +(-\phi)^{n})| &=& \prod_{i=1}^{k} \prod_{r=N_i+1}^{N_i + F_{n_i}}2|\sin \pi(r\phi + (-\phi)^{n})| \nonumber \\
	&=& \prod_{i=1}^{k} \prod_{r=1}^{F_{n_i}} 2|\sin\pi(r\phi + N_i\phi +(-\phi)^{n}) | \label{finitely}\\
	&=& \prod_{i=1}^{k} \prod_{r=1}^{F_{n_i}} 2\left|\sin\pi\left(r\phi + \frac{(-1)^{n_i+1} \ve_i}{F_{n_i}}\right)\right| \nonumber ,
\end{eqnarray}
where for $ i=1,2,\ldots, k$ we define $\ve_i$ so that 
$$ \frac{ (-1)^{n_i+1} \ve_i}{F_{n_i}} =   (N_i\phi + (-\phi)^n) \underbrace{- F_{n_{i+1}-1} - F_{n_{i+2}-1} - \dots - F_{n_{k-1} - 1} - F_{n_k - 1}}_{\in \mathbb{Z}}.
$$
The integer which is subtracted in the formula above is chosen in such a way that $\ve_i$ is small (see below). Note that subtracting this integer is possible without problems by the periodicity of the sine-function. Comparing with the definition of the perturbed Sudler products in \eqref{sudlereps} we note that we have
\begin{equation} \label{prodl}
\prod\limits_{r=1}^{F_n-N-1} 2|\sin\pi ( r \phi +(-\phi)^{n})| = \prod_{i=1}^{k}P_{F_{n_i}}(\phi, \ve_i).
\end{equation}
Furthermore, according to \eqref{phi1} we have
\begin{eqnarray*}
 \frac{\ve_i}{(-1)^{n_i+1}F_{n_i}}	& = & \left(N_i\phi + (-\phi)^{n}\right) - F_{n_{i+1}-1} - \ldots - F_{n_k - 1} \\
& = &  (F_{n_{i+1}}+\ldots + F_{n_{k}})\phi + (-\phi)^{n}  - F_{n_{i+1}-1} - \ldots -F_{n_k - 1} \\
& = & - \big((-\phi)^{n_{i+1}} + \ldots + (-\phi)^{n_{k}} - (-\phi)^n \big).
\end{eqnarray*}
Thus 
\begin{eqnarray}
\frac{\ve_i}{F_{n_i} \phi^{n_i}} & = & \phi^{-n_i} (-1)^{-n_i} \left( (-\phi)^{n_{i+1}} 
+ \ldots + (-\phi)^{n_{k}} 
- (-\phi)^n \right) \label{as_in_line}\\
& = & (-1)^{n_{i+1}-n_i} \phi^{n_{i+1}-n_i} + \ldots + (-1)^{n_k-n_i} \phi^{n_{k}-n_i} 
- (-1)^{n - n_i} \phi^{n-n_i} \label{pos_to}
\end{eqnarray}
for all $i$. In the last line, we (in general) have positive as well as negative summands. The positive terms are those coming from indices having the same parity mod $2$ as $n_i$, while the other terms give negative contributions. Note, however, that by the properties of the Zeckendorff representation (which cannot have two consecutive digits equal to $1$) we necessarily have $n_{i+1} - n_i \geq 2$; consequently, the maximal possible negative contribution is smaller than the maximal possible positive contribution, since $n_{i+1}$ and $n_i$ can only have different parity mod 2 if $n_{i+1} - n_i \geq 3$. This implies that the maximal possible positive contribution to \eqref{pos_to} is bounded above by
\begin{equation} \label{range_1}
\phi^2 + \phi^4 + \phi^6 + \dots = \phi,
\end{equation}
while the maximal possible negative contribution is bounded by
\begin{equation} \label{range_2}
-\phi^3 - \phi^5 - \phi^7 - \dots = - \phi^2.
\end{equation}

By \eqref{phi2} we have  $F_{n_i}\phi^{n_i} \sim \frac{1}{\sqrt{5}}$ as $n_i\to \infty$. Thus we have
$$
- \phi^2/\sqrt{5} - 0.001 \leq \ve_i \leq \phi / \sqrt{5} + 0.001,
$$
whenever $n_i \geq i_0$ (here and in the sequel we write $i_0$ for generic absolute lower bounds for elements of the index set, not necessarily the same at different occurrences). By $\phi^2/\sqrt{5} \approx 0.171$ and $\phi / \sqrt{5} \approx 0.276$ this implies that for $n_i \geq i_0$ we have
$$
- 0.18 \leq \ve_i \leq 0.28.
$$
Consequently, by Proposition \ref{prop1} we have $G(\ve_i)>1.01$ for all $n_i \geq i_0$, which implies that $P_{F_{n_i}}(\phi,\ve_i) \geq 1$ for $n_i \geq i_0$ (recall here that the convergence towards $G$ is uniform in $\ve$). Thus in the product on the right-hand side of \eqref{prodl} all factors are at least $1$, except for the contribution of a finite number of small indices $n_i \leq i_0$. \\

The contribution of the finitely many indices with $n_i \leq i_0$ to the full product, in the decomposition in line \eqref{finitely}, is
\begin{eqnarray}
& & \prod_{\substack{1 \leq i \leq k,\\n_i \leq i_0}} 
\prod_{r=1}^{F_{n_i}} 2|\sin\pi(r\phi + N_i\phi +(-\phi)^{n})|. \label{the_term}
\end{eqnarray}
Let $j = \max\{n_i:~n_i \leq i_0\}$, and assume that $j$ is odd (the even case is perfectly analogous). Also define $(\!(x)\!) = x$ for $x\in(-1/2,1/2]$ and extend to all $x\in\mathbb{R}$ with period $1$.  We can write the number $N_i + r$ as $y + z$, where the Zeckendorff representation of $y$ contains only Fibonacci numbers of size at least $F_{j+2}$, and where that of $z$ only contains Fibonacci numbers of size at most $F_j$ (note that $F_{j+1}$ cannot occur at all, since we know that $F_j$ does occur). By the best approximation properties of continued fraction convergents, we can easily deduce that $(\!( z \phi)\!) \not\in (-\phi^{j},\phi^{j+1})$, with the left and right endpoints of this interval corresponding to $z=F_j$ and $z=F_{j-1}$, respectively. Recall here that we assumed that $j$ is odd, so that $(\!(F_j z)\!)$ is positive and $(\!(F_{j-1} z)\!)$ is negative by \eqref{phi1}. On the other hand, using \eqref{phi1} and arguing as in the lines leading to \eqref{range_1} and \eqref{range_2}, we have
$$ (\!(y \phi + (-\phi)^n )\!)  \in  \left( - \sum_{\substack{\ell \geq j+2,\\ \ell \text{~even}}} \phi^{-\ell},~ \sum_{\substack{\ell \geq j+2,\\ \ell \text{~odd}}} \phi^{-\ell} \right) = \left( - \phi^{j+2}, \phi^{j+1} \right).
$$
Thus we have
$$ (\!( r\phi + N_i\phi +(-\phi)^{n} )\!)  = 
(\!( (y+z) \phi + (-\phi)^{n} )\!)  
\not\in \left(-\phi^{j} + \phi^{j+1}, \phi^{j+1} - \phi^{j+2} \right).
$$
Noting that trivially $j \leq i_0$, we deduce that for every individual factor appearing in one of the products in \eqref{the_term}, the term 
$|\sin\pi(r\phi + N_i\phi +(-\phi)^{n})|$ is bounded below by an absolute constant (depending only on $i_0$). Since the product \eqref{the_term} contains a bounded number of factors, we can deduce that this product is bounded below by an absolute constant.\\

Thus we can conclude that there exists an absolute constant $K>0$ such that 
$$ 
\liminf_{N\to\infty}\hspace{-1mm}\prod_{r=1}^{F_n-N-1}
\hspace{-3mm}2|\sin\pi ( r \phi + (-\phi)^{n})| \geq K.  
$$
Finally we deduce 
$$ \limsup_{N\to\infty}\frac{P_N(\phi)}{N} = \limsup_{N\to\infty} \frac{F_n}{N}\cdot \frac{P_{F_n-1}(\phi)}{F_n}\cdot 
\left(  \prod_{r=1}^{F_n-N-1}\hspace{-3mm}2|\sin\pi ( r \phi + (-\phi)^{n})| \right)^{-1} <   \infty.  $$
\end{proof}
For illustration, we show that a proof for the fact that $\liminf P_N(\phi)>0$ can be obtained in a way which is completely analogous to the one above, just without the ``reflection'' at $F_n$. Let $N$ be given. We expand $N$ into its Zeckendorff representation
 $$
 N = F_{n_k} + F_{n_{k-1}} + \ldots + F_{n_1}.
 $$
 With $N_i$ defined as in the proof of Theorem \ref{th3}, we have
 \begin{eqnarray*}
 	P_N(\phi) & = & \prod_{r=1}^N 2 |\sin \pi r \phi| \\
 	& = & \prod_{i=1}^k \prod_{r=1}^{F_{n_i}} 2 |\sin\pi (r \phi + N_i \phi)|.
 \end{eqnarray*}
 This is very similar to line \eqref{as_in_line}, except that the term $\phi^n$ is missing. We can write $P_N(\beta) = \prod_{i=1}^k P_{F_{n_i}} (\phi,\ve_i)$ for some appropriate perturbations $\ve_i$. Since the term $\phi^n$ is now missing, the perturbations $\ve_i$ are slightly different. However, the possible \emph{range} for these perturbations is exactly the same as above, since $\phi^n$ is just one term of a geometric progression, and we have estimated the possible range of perturbations by considering the whole infinite geometric progressions -- cf.\ \eqref{range_1} and \eqref{range_2}. So we obtain the same range of perturbations as previously, i.e.\ $\ve_i \in [-0.18,0.28]$ for all $i$, and throughout this range of possible perturbations the function $G$ is uniformly bounded below by 1.01 (see above). This allows us to deduce that $\liminf\limits_{N \to \infty} P_N(\phi) > 0$.

 \begin{proof}[Proof of Proposition \ref{prop1}]
 	
 	As in the proof of Theorem \ref{th2}, we write 
 	\begin{eqnarray*} 
 		1 &=& \lim_{n\to\infty} \frac{P_{F_{n+2}}(\phi)}{P_{F_{n+1}}(\phi)} = \lim_{n\to\infty} \prod_{r=F_{n+1}+1}^{F_{n+2}}\!\!\!2|\sin \pi r\phi |  =  \lim_{n\to\infty}\hspace{-3mm}\prod_{r=1}^{F_{n+2}-F_{n+1}}\hspace{-3mm}2|\sin\pi(r+ F_{n+1})\phi|   \\
 		&\stackrel{\eqref{phi1}}{=}& \lim_{n\to\infty}\prod_{r=1}^{F_{n}}2| \sin\pi(r\phi + (-1)^{n}\phi^{n+1})| = \lim_{n\to\infty}P_{F_n}(\phi, -F_n\phi^{n+1}) .		
 	\end{eqnarray*}
 	Here for the sequence $\zeta_n := -F_n\phi^{n+1}, n=1,2,\ldots$ we have $\lim \zeta_n = \zeta_0 := -\phi/\sqrt{5}$ by \eqref{phi2}, and since the functions $P_{F_n}(\phi,\ve)$ converge uniformly to $G(\ve)$ we deduce that 
 	$$  1 = \lim_{n\to \infty}P_{F_n}(\phi,\zeta_n) = G\left( -\phi/\sqrt{5} \right)  \, . $$
 
 	To prove the remaining assertions of the theorem, we let $I \subseteq \mathbb{R}$ be a closed interval such that $G(\ve) \neq 0$ for all $\ve \in I$. Note that $\log G$ is well defined in the set of $\ve$ for which $G(\ve)\neq 0$. In view of \eqref{gformula}, to show that $G$ is a $C^{\infty}$ function on $I$, it is enough to establish the uniform convergence of
 	\[ H_{R}\left(\varepsilon\right)=\log\prod_{r\leq R}
 	\vert 1-\delta_{r}^{2}(\ve) \vert, \qquad \mathrm{where}\quad
 	\delta_{r}(\ve)= \frac{2\sqrt{5}\varepsilon+1}{u(r)} \,\, \text{ for any } \varepsilon\in I  \]
 	and its derivatives of any order as $R\to \infty$. First, we note that since $I$ is compact, the distance  $ \min\{ \left|1- \delta_{r}(\ve) \right| : \ve\in I, r\geq 1\}$ is strictly positive. Furthermore, the estimate 
 	\[  0\geq\log\left|1-\delta\right|
 	= -\int_{1-\delta}^{1}\frac{\mathrm{d}x}{x}\geq
 	-\int_{1-\delta}^{1}\frac{\mathrm{d}x}{1-\delta}
 	=-\frac{\delta}{1-\delta} \, , \hspace{5mm} 0< \delta < 1
 	\]
 	implies  that for $r$ large enough, 
 	\begin{equation*}
 	\vert \log\vert 1-\delta_{r}^{2}(\ve) \vert \vert  \leq \frac{\delta_r^2(\ve)}{1-\delta_r^2(\ve)} \ll_{I} \frac{(2\sqrt{5}\varepsilon+1)^{2}}{ u(r)^{2}}
 	\ll_{\, I} \, \frac{1}{r^{2}} \cdot  \label{eq: supnorm bound G}
 	\end{equation*}
 	Thus, the partial sums $ H_R(\ve)=\sum\limits_{r\leq R} \log\vert 1-\delta_{r}^{2}(\ve) \vert $ 
 	converge uniformly to
 	\begin{equation*}
 	H(\ve)  =  \sum_{r=1}^{\infty} \log\vert 1-\delta_{r}^{2}(\ve) \vert, \hspace{5mm} \ve\in I .
 	\end{equation*}
 	\par We proceed to show that the derivatives $H_R^{(k)}(\ve)$ of any order $k\geq 1$ converge uniformly, too.  For $k\geq 1$, the $k$-th derivative of each summand of $H_R(\ve)$ is 
 	\begin{align*}
 	\frac{\mathrm{d}^{k}}{\mathrm{d}\varepsilon^{k}}
 	\log\vert 1- \delta_{r}^{2}(\ve)\vert 
 	& =  \frac{(-1)^{k-1} (2\sqrt{5})^k (k-1)! }{u(r)^k} 
 	\Bigg(  \frac{1}{\left( 1+ \delta_{r}(\ve)   \right)^k}  + \frac{(-1)^k}{\left( 1-  \delta_{r}(\ve)   \right)^k}    \Bigg).
 	\end{align*} For the case $k=1$, we infer that
 	\[
 	\frac{\mathrm{d}}{\mathrm{d}\varepsilon}
 	\log \vert 1 - \delta_{r}^{2}(\ve) \vert \, =\,  \frac{4\sqrt{5}(2\sqrt{5}\varepsilon+1)}{u\left(r\right)^2 (1-\delta_r^2(\ve)) } \,
 	\ll_I \,  \frac{1 }{u\left(r\right)^{2}}\, \ll \,  \frac{1}{r^{2}} \, ,
 	\]
 	which is summable over $r$.
 	Now we observe that for each $k\geq 2$ fixed, 
 	\[ \left| u(r)^{k}\left(1\pm \delta_{r}(\ve) 
 	\right)^{k}\right|\, \gg_I \,  u(r)^{k}
 	\, \gg \,  r^{k} ,
 	\]
 	and this implies that the $k$-th derivative of $H_R$
 	converges uniformly on $I$.  The upshot is that, for
 	any $k\geq1$, the partial sums of
 	\[  \sum_{r=1}^{\infty}\frac{\mathrm{d}^{k}}
 	{\mathrm{d}\varepsilon^{k}}\log \vert 1 - \delta_{r}^{2}(\ve) \vert 
 	\]
 	converge uniformly on $I$ to $H^{(k)}$. It follows that $G$ is $C^{\infty}$ in any interval where it is nonzero, and by its definition we deduce that it is also continuous on $\mathbb{R}$.
 	\par Having established that $G\left(\varepsilon\right)$ is $C^{\infty}$,
 	except in the discrete set of its roots, makes demonstrating
 	the concavity assertion now easier. Indeed, it suffices to
 	argue that the second derivative of $\log G\left(\varepsilon\right)$ is
 	strictly negative.  In the previous part of the proof, we have actually shown that
 	\begin{equation} \label{logG}
 	\log G(\varepsilon) = \log |\sqrt{5}\varepsilon+1| + H(\varepsilon) + A,
 	\end{equation} 
 	where $A$ is an absolute constant. Hence, the second derivative of $\log G(\varepsilon)$ is 
 	$$- 5 (\sqrt{5}\ve +1)^{-2} -20 \sum_{r=1}^{\infty}(\vert u(r)
 	-2\sqrt{5}\varepsilon-1 \vert ^{-2}
 	+ \vert u(r)+2\sqrt{5}\varepsilon+1 \vert ^{-2})\, <\, 0,  $$
 	which implies the log-concavity of $G$ on the intervals under consideration.

 	Recall that $G(-\phi/\sqrt{5})=1$, and note that $-\phi/\sqrt{5} < -0.26$. For the values $\ve=-0.26$ and $\ve = 0.58$ we can prove that  $G(\ve)>1$ (or, equivalently, that $\log G(\ve)>0$). Indeed, we can explicitly estimate the error we make when we approximate the infinite series $H(\ve)$ in \eqref{logG} by a finite series -  the tail behavior of this infinite series is essentially the same as that of $\sum_{r=1}^{\infty} 1/r^2$. Accordingly, we can calculate that $G(-0.26) \in [1.09, 1.11]$  and that $G(0.58) \in [1.10, 1.12]$, and in particular that $G(-0.26)>1.01$ and $G(0.58)>1.01$. The log-concavity of $G$ implies that actually $G(\ve)>1.01$ holds throughout the whole range $\ve \in (-0.26,0.58)$.

 \end{proof}
 
 \section{Proofs of Theorems for the  quadratic irrationals $\be=[b,b,\ldots]$ }
 \begin{proof}[Proof of Theorem \ref{th6}]
 	As one might expect, the proof goes entirely along the lines of the proof of Theorem \ref{th1}. Given a fixed integer $n\geq 1$, we write $[k]$ for the residue of the integer $k\in\mathbb{N}$ modulo $q_n$. Define 
 	\begin{eqnarray*}
 		s_n(0,\ve) &=& 2\sin\pi\left(\frac{\ve}{q_n}+\frac{\be^{n+1}}{2} \right), \\
 		s_n(r) &=&  2\sin\pi\left( \frac{r}{q_n}+\be^{n+1}\left(\frac{[q_{n-1}r]}{q_n} -\frac{1}{2}\right)\right), \, r=1,2,\ldots, F_n-1.
 	\end{eqnarray*}
 	The product $P_{q_n}(\be,\ve)$ is further factorized as
 	\begin{equation} \label{bfactorisation} 
 	P_{q_n}(\be,\ve) = A_n(\be,\ve)\cdot B_n(\be)\cdot C_n(\be,\ve),
 	\end{equation}
 	where   
 	\begin{eqnarray*}
 		A_n(\be,\ve) &=& 2 q_n\left| \sin \pi\left(q_n\be + (-1)^n\frac{\ve}{q_n} \right)\right|,\\
 		B_n(\be) &=& \prod_{r=1}^{q_n-1} \frac{s_n(r)}{2\left| \sin(\pi r/q_n)\right|}, \hspace{3mm} \text{ and} \\
 		C_n(\be, \ve) &=& \prod_{r=1}^{q_n-1}\left( 1-\frac{s_n(0)^2}{s_n(r)^2} \right)^{\frac{1}{2}} .
 	\end{eqnarray*}
 	Using \eqref{b1}, \eqref{b2} together with the asymptotic estimate $\sin x \sim x, x\to 0$ we get
 	\begin{eqnarray*} 
 		A_n(\be,\ve) &=& 2q_n|\sin\pi((-1)^n\be^{n+1} +(-1)^n\frac{\ve}{q_n} )| \\ 
 		&=& 2q_n\left|\sin\pi\left( \be^{n+1} + \frac{\ve}{q_n}\right)\right|  \\
 		&\sim & 2\pi q_n \left|\be^{n+1} + \frac{\ve}{q_n} \right| \\
 		&\sim & 2\pi\left| \ve + \frac{1}{\sqrt{b^2+4}}\right|, \qquad n\to\infty .
 	\end{eqnarray*}
 	
 Regarding the products $B_n(\be)$, it is shown in \cite{sigrid2} that there exists a constant $B_b>0$ such that $B_n(\be)\to B_b$ as $n\to \infty$ and finally for the factor $C_n(\be,\ve)$ we can show arguing as in \cite{mv} that 
 	$$ \lim_{n\to\infty}C_n(\phi,\ve) = \prod_{r=1}^{\infty}\left(1-\frac{(2\ve\sqrt{b^2+4}+1)^2}{u_b(r)^2} \right)  \hspace{3mm} \text{ for all } \ve\in\mathbb{R}, $$
 where $u_b(r)$ was defined in \eqref{ubr}. Combining these facts for the asymptotic behavior of the factors in \eqref{bfactorisation}  we obtain the requested convergence result for $P_{q_n}(\be,\ve)$.  The fact that the convergence is uniform can be established as in the proof of Theorem \ref{th1}.
 \end{proof}
 
 \begin{proof}[Proof of Theorem \ref{cbvaluethm}]
 	We calculate
 	\begin{eqnarray*}
 		\frac{P_{q_{n+1}}(\be)}{P_{q_{n-1}}(\be)}	&=&\! \prod_{r=q_{n-1}+1}^{q_{n+1}}\!\!\!2|\sin \pi r \be| = \prod_{r=1}^{bq_n}2| \sin \pi(r+q_{n-1})\beta | \\
 		&\stackrel{\eqref{b1}}{=}& \prod_{r=1}^{bq_n} 2| \sin \pi(r\beta + (-1)^{n-1}\be^n)| \\
 		&=& \prod_{j=1}^{b}\prod_{r=(j-1)q_n+1}^{jq_n}\!\!2|\sin\pi(r\be -(-\be)^n)| \\
 		&=& \prod_{j=1}^{b}\prod_{r=1}^{q_n}2|\sin\pi(r\beta + (j-1)q_n\be -(-\be)^n)| \\
 		&\stackrel{\eqref{b1}}{=}& \prod_{j=1}^{b}\prod_{r=1}^{q_n}2|\sin \pi(r\be + (-1)^{n}(j-1)\be^{n+1} -(-\be)^n) |  \\
 		&=& \prod_{j=1}^{b}P_{q_n}(\be, \venj) ,
 	\end{eqnarray*}	
 	where for $j=1,2,\ldots,b$ we have set 
 	\begin{eqnarray*}
 		\venj & :=& q_n\be^n [(j-1)\be -1] \\
 		&=& q_n\be^{n+1} [(j-1)\be-1]\be^{-1} \\
 		& \sim & - \frac{b-j + \be +1}{\sqrt{b^2+4}},\hspace{4mm} n\to \infty.
 	\end{eqnarray*}	
 	At this point we may once again factorise 
 	$$ P_{q_n}(\be, \venj) = A_n(\be,\venj) \cdot B_n(\be) \cdot C_n(\be,\venj), \qquad j=1,2,\ldots,b   $$
 	where the factors are as in \eqref{bfactorisation}. Now
 	\begin{eqnarray*}
 		A_n(\be, \venj) &=& 2q_n| \sin \pi(q_n\beta + (-1)^n\frac{\venj}{q_n}  )|  \\
 		&=& 2q_n |\sin \pi(\be^{n+1} + \be^n [(j-1)\be -1] )  | \\
 		& \sim & 2\pi q_n \be^{n+1} (\be + b - j) \\
 		&\sim & \frac{2\pi (\be + b - j)}{\sqrt{b^2+4}}, \hspace{4mm} n\to \infty .
 	\end{eqnarray*}
 	As mentioned before,  it is shown in \cite{sigrid2} that there is $B_b>0$ such that $\lim\limits_{n\to\infty}B_b(\phi)=B_{b}$. \vspace{-1.2mm} \newline Finally regarding the third factor we notice that
 	\begin{eqnarray*}
 		s_n(0,\venj) &=& 2|\sin \pi\left(\frac{\venj}{q_n}+ \frac{\be^{n+1}}{2}  \right)|  \\
 		&=& 2| \sin \pi (\be^{n+1} [j-1-\be^{-1}] + \frac{\be^{n+1}}{2}) | \\
 		&\sim & \pi \be^{n+1} ( 2b- 2j +2\be +1), \hspace{4mm} n \to \infty
 	\end{eqnarray*}
 	and 
 	$$ s_n(r) \sim  2\pi \be^{n+1}\sqrt{b^2+4} \left( r - \frac{\{r\be\}-\frac{1}{2}}{\sqrt{b^2+4}}\right) = \pi \beta^{n+1}u_b(r)  , \hspace{3mm} \, n\to\infty $$
 	hence repeating once again the arguments utilised in \cite{mv, sigrid2} we deduce that
 	$$\lim_{n\to\infty} C_n(\be, \venj) = \prod_{r=1}^{\infty}\left(1 - \frac{(2b-2j+2\be+1)^2}{u_b(r)^2} \right)  $$
 	Combining, 
 	$$ \lim_{n\to\infty}P_{q_n}(\be, \venj) = \frac{2\pi(\be+b-j)}{\sqrt{b^2+4}} \cdot B_{(b)} \cdot  \prod_{r=1}^{\infty}\left(1 - \frac{(2b-2j+2\be+1)^2}{u_b(r)^2} \right) . $$
 	Now by Theorem A $2$ we may deduce 
 	$$ 1 = \lim_{n\to\infty} \frac{P_{q_{n+1}}(\be)}{P_{q_{n-1}}(\be)} = \lim_{n\to\infty}\prod_{j=1}^{b}P_{q_n}(\be, \venj), $$
 	whence 
 	\begin{equation} \label{cbeq1}
 	1 = \left( \frac{2\pi}{\sqrt{b^2+4}}  \right)^b \cdot \beta(\beta+1)\cdots (\beta + b-1) \cdot B_{b}^b \cdot \prod_{r=1}^{\infty} \prod_{j=1}^{b} \left(1 - \frac{(2b-2j +2\be +1)^2}{u_b(r)^2}  \right) \, .
 	\end{equation}
 	At this point we observe that in the proof of Theorem A $2$ in \cite{sigrid2} it is actually shown that the constant $C_b>0$ satisfies
 	\begin{equation} \label{cbeq2}
 	C_b = \frac{2\pi}{\sqrt{b^2+4}} \cdot B_{b} \cdot \prod_{r=1}^{\infty} \left( 1 - \frac{1}{u_b(r)^2}\right) .
 	\end{equation}
 	Combining \eqref{cbeq1} with \eqref{cbeq2} we deduce the value of $C_b$. 
 \end{proof}
 
 We now formulate and prove some complementary results that will be used later in the proof of Theorem \ref{lineargrowththeorem}.
 
 \begin{prop} \label{prop: MV constant for b at least 15} If $b\geq 11$, then $C_{b}<1$.
 \end{prop}
 \begin{proof}
 	We have $ \be \geq (b + b^{-1})^{-1}$ and $(\be + 1)\cdots (\be+ b-1) \geq  (b-1)!$. Thus by \eqref{cbvalue},\begin{eqnarray*}
 		C_b^b & \leq&  \frac{b + b^{-1}}{  (b-1)!} \prod_{r=1}^{\infty}\prod_{j=1}^{b} \left( 1-  \frac{1}{u_b(r)^2}\right)\left( 1 -  \frac{(2j +2\be -1)^2}{u_b(r)^2} \right)^{-1}  \\
 		& = &   \frac{b+b^{-1}}{ (b-1)!} \prod_{r=1}^{\infty}\prod_{j=1}^{b} \left( 1 + \frac{(2j+2\be -1)^2 -1 }{u_b(r)^2 - (2j + 2\be -1)^2 } \right) .
 	\end{eqnarray*}
 	Now using the estimates $1+x \leq e^x, x\in\mathbb{R}$ and  $u_b(r)^2-(2b+2\be-1)^2 \geq 2(b^2+4)r^2, r\geq 1$ we find
 	\begin{eqnarray*}
 		C_b^b & \leq &  \frac{b+ b^{-1}}{ (b-1)!} \prod_{r=1}^{\infty}\prod_{j=1}^{b} \exp\left(   \frac{(2j+2\be -1)^2 -1 }{u_b(r)^2 - (2j + 2\be -1)^2 } \right) \\
 		& \leq &   \frac{b+ b^{-1}}{ (b-1)!} \exp\left( \sum_{r=1}^{\infty}\sum_{j=1}^{b}   \frac{(2j+2\be -1)^2 -1 }{u_b(r)^2 - (2b + 2\be -1)^2 }   \right) \\
 		& \leq &  \frac{b+ b^{-1}}{ (b-1)!} \exp\left( \sum_{r=1}^{\infty}  \frac{ \frac{4}{3}b^3 + 4\be b^2 - (\frac{4}{3}-4\be^2 )b }{2(b^2+4)r^2 }  \right) \\
 		&\leq &   \frac{b+ b^{-1}}{ (b-1)!} \exp\left( \frac{(2b^3 + b + 6b^{-1} )\pi^2 }{18(b^2 + 4)}  \right) . 
 	\end{eqnarray*}
 	The right hand side is decreasing as a function of $b$ for $b\geq 11$, and we can verify that it is less than $1$ when $b=11$. This proves that $C_b<1$ for all $b\geq 11$. 
 \end{proof}

 \begin{cor}\label{cor: when M-V const greater one}
 	The constant $C_{b}$ exceeds $1$ if and only if $b\in \{1, 2, 3, 4, 5, 6\}$.
 \end{cor}
 \begin{proof} By \eqref{cbvalue}, $\log C_b$ is given by an infinite series, and we can provide explicit estimates for the error of approximation of the series by a finite sum. 
 	By Proposition \ref{prop: MV constant for b at least 15}, in order to prove the corollary 
 	it remains to compute the value of $C_{b}$ for $1\leq b \leq 10$. These values are shown in the following table, with precision of $6$ decimal digits. \\
 	
 	\noindent\begin{minipage}[c]{1\columnwidth}%
 		\begin{center}
 			\begin{tabular}{|c|c|c|c|c|c|}
 				\hline 
 				$b$ & $1$ & $2$ & $3$ & $4$ & $5$ \tabularnewline
 				\hline 
 				$C_{b}$ & $2.406152$ & $2.159658$ & $1.800517$ 
 				& $1.499350$ & $1.267273$ \tabularnewline
 				\hline 
 				\hline 
 				$b$ & $6$ & $7$ & $8$ & $9$ & $10$ \tabularnewline
 				\hline 
 				$C_{b}$& $1.089429$ & $0.951175$ & $0.841663$ & $0.753296$ & $0.680773$ 
 				\tabularnewline
 				\hline 
 			\end{tabular}
 			\par\end{center}%
 	\end{minipage}\\~\\
 	The corollary is now proved.
 \end{proof}

 The proof of Theorem \ref{lineargrowththeorem} is given in several lemmas, which deal with different cases for the value of the integer $b$. We make use of the \emph{Ostrowski expansion} of integers with respect to some given irrational number, which we now briefly present. 
 Fix some irrational number $\alpha\in (0,1)$ with continued fraction expansion $\alpha=[a_1, a_2, \ldots]$ and let $(q_n)_{n=0}^{\infty}$ be 
 the corresponding sequence of denominators. 
 Every positive integer $N\geq 1$ can be written uniquely in the form 
 \begin{equation} \label{ostrowski}
 N = c_{n+1}q_n + c_{n}q_{n-1} + \ldots + c_2q_1 + c_1q_0
 \end{equation}
 where the integers $(c_i)_{i=1}^{n+1}$ are such that  
 \begin{enumerate}
 	\item[$1.$] $0\leq c_1 < a_1$  
 	\item[$2.$] $0\leq c_{i+1} \leq a_{i+1}$ \text{ for } $i=1,\ldots, n$   
 	\item[$3.$] $c_i=0$ \text{ whenever } $c_{i+1}=a_{i+1}$.  
 \end{enumerate}
 The expansion of $N$ as in \eqref{ostrowski} 
 is called the {\it Ostrowski expansion} of $N$ 
 with respect to $\alpha$. For more details we refer 
 to \cite[Chapter II]{rocket}. \\

 \begin{lemma} \label{lemma>1}
 	Let $\be$ be defined as in Theorem \ref{lineargrowththeorem}. If $b \geq 7$, then 
 	$$\liminf \limits_{N \to \infty} P_N(\be) = 0  \qquad \text{and} \qquad \limsup \limits_{N\to\infty}\dfrac{P_N(\be)}{N} = \infty. $$
 \end{lemma}
 \begin{proof} 
 	Let $b \geq 7$ be fixed, and let $\be$ be defined as in Theorem \ref{lineargrowththeorem}. The key fact in the proof will be that $C_b <1$. By the upper bound for $C_b$ given in the proof of Proposition \ref{prop: MV constant for b at least 15} and the table in Corollary \ref{cor: when M-V const greater one} we see that we actually have $C_b \leq 0.96$ for $b\geq 7$.  By Proposition \ref{gbetaproperties}, $G_\beta $ is continuous near $\ve=0$, so there exists $\eta>0$ small enough such that $G_\beta(\ve) \leq 0.98$ for $|\ve| \leq \eta$. \par Set $m_1=1$. We construct a sequence $m_1 < m_2 < m_3 < \dots$ of positive integers, such that:
 	\begin{enumerate}
 		\item[(i)] every $m_k$ is the denominator of a convergent to $\beta$, 
 		\item[(ii)] $m_k \geq 2 m_{k-1}$ for all $k \geq 2$, and 
 		\item[(iii)] $\|m_k \beta\| < \eta / (4 m_{k-1})$ for $k \geq 2$.
 	\end{enumerate} 
 	This construction is always possible by choosing $m_k$ sufficiently large  compared to $m_{k-1}$.\par
 	
 	Let $N_k = m_k + m_{k-1} + \dots + m_1, \, k\geq 1$. Then 
 	\begin{eqnarray}
 	P_{N_k}(\beta)\, & = \,& \prod_{r=1}^{N_k} 2|\sin \pi r\beta| \,\, = \,\, \prod_{j=1}^{k}\prod_{r=1}^{m_j}2|\sin \pi (M_j + r) \beta| \nonumber\\
 	& = \, & \prod_{j=1}^{k} \prod_{r=1}^{m_j}2 \left|\sin \pi\left(r \beta + \frac{\ve_j}{m_j}\right)\right| \,\, = \,\, \prod_{j=1}^{k}P_{m_j}(\beta,(-1)^j\ve_j) , \label{prod_nk}
 	\end{eqnarray}
 	where
 	\begin{equation} \label{mj}
 	M_k = 0 \qquad \text{and} \qquad M_j = m_k + m_{k-1} + \dots + m_{j+1}, \quad j=1,2,\ldots, k-1
 	\end{equation}
 	and where either 
 	$$
 	\ve_j =m_j \|  M_j \beta\| \qquad \text{or} \qquad \ve_j = -m_j \| M_j \beta\|.
 	$$
 	In any case we have
 	\begin{eqnarray}
 	|\ve_j| & = & m_j \|(m_{j+1} + m_{j+2} \dots + m_{k}) \beta\| \nonumber\\
 	& \leq & m_j (\|m_{j+1} \beta\| + \| m_{j+2} \beta \| + \dots + \|m_{k} \beta\|) \label{anycase}\\
 	& \leq & \frac{\eta}{4} \left( \frac{m_j}{m_j} + \frac{m_j}{m_{j+1}} + \dots + \frac{m_j}{m_{k}} \right) \nonumber\\ 
 	& \leq & \frac{\eta}{2}. \label{anycase2}
 	\end{eqnarray}
 	Thus we have $G_\beta(\ve_j) \leq 0.98$ for all $j$, and consequently  $P_{m_j}(\beta,\ve_j) \leq 0.99$ for all sufficiently large $j$. Accordingly, in \eqref{prod_nk} all factors (except for finitely many) are smaller than 0.99 and  the product tends to 0 as $k \to \infty$. This proves the first assertion of the lemma.\\
 	
 	For the second part, we set $N_k = m_{k+1} - 1 - m_k - m_{k-1} - \ldots - m_1,\, k\geq 1$. Then
 	\begin{equation} \label{pnkbeta}
 	P_{N_k}(\beta) = \frac{P_{m_{k+1}-1}(\beta)}{\prod\limits_{r=N_k+1}^{m_{k+1}-1}\!\!2|\sin \pi r \beta|} = \frac{P_{m_{k+1}-1}(\beta)}{\prod\limits_{j=1}^{k}\prod\limits_{r=1}^{m_j} 2|\sin \pi (-m_{k+1} + M_j + r) \beta|},
 	\end{equation}
 	with $M_j$ defined as in \eqref{mj}. By Theorem A 2 the numerator in \eqref{pnkbeta} satisfies
 	$$ P_{m_{k+1}-1}(\beta) \, \, \asymp \,\, m_{k+1} \, \asymp \, \,  N_k\, , \qquad k\to \infty,$$
 	and it remains to show that the denominator in \eqref{pnkbeta} tends to zero as $k \to \infty$. This can be done in exactly the same way as above, the only difference being that in \eqref{anycase} we now have an additional term $m_j \|m_{k+1} \beta\|$, which does not affect the validity of \eqref{anycase2}. Thus $\prod_{r=N_k+1}^{m_{k+1}-1} 2|\sin \pi r \beta| \to 0$ as $k \to \infty$, and $P_{N_k}(\beta) / N_k \to \infty$ as $k \to \infty$, which proves the second part of the lemma.
 \end{proof}
 
 \noindent The following lemma settles the case $b=6$ in Theorem \ref{lineargrowththeorem}.
 
 \begin{lemma} \label{lemma_b6}
 	Let $\be$ be defined as in Theorem \ref{lineargrowththeorem}, for $b=6$; that is, $\beta = \sqrt{10}-3$. Then 
 	$$\liminf \limits_{N \to \infty} P_N(\be) = 0 \qquad \text{and} \qquad \limsup \limits_{N\to\infty}\dfrac{P_N(\be)}{N} = \infty. $$
 \end{lemma}

 It was already noted in \cite{gkn} that $P_N(\beta)$ seems to be decreasing along a subsequence of indices $N$. In \cite{gkn} the following table concerning the evolution of minima of $P_N(\beta)$ is given:\\
 
 \begin{tabular}{|l|l|l|l|l|l|l|}
 	\hline
 	$N$ & 1 & 7 & 44 & 272 & 1677 & 10335 \\
 	\hline
 	$P_N(\beta)$ & 0.977 & 0.907 & 0.849 & 0.794 & 0.742 & 0.693\\
 	\hline
 \end{tabular}
 ~\\
 
 The denominators of continued fraction convergents to $\beta$ are $1,6,37,228,1405,8658, \dots$, and one can see that the indices $N$ above arise as sums of such denominators. For example, we have $272 = 228 + 37 + 6 + 1$, or $1677 = 1405 + 228 + 37 + 6$. We will exploit this structure to construct a subsequence of indices along which the Sudler product at $\beta$ tends to zero.\\
 
 A similar table shows values of $N$ for which the ratio $P_N(\beta)/N$ is large.\\
 
 \begin{tabular}{|l|l|l|l|l|}
 	\hline
 	$N$ & 30 & 184 & 1133 &6981 \\
 	\hline
 	$P_N(\beta)$ & 1.061 & 1.213 & 1.286 & 1.378 \\
 	\hline
 \end{tabular}
 ~\\
 
 From the analogue of the reflection principle \eqref{reflectionprinciple} for $\beta$, it is not surprising that these indices $N$ also arise from denominators of convergents to $\beta$, but in a ``reflected'' way. Indeed, we can see that $30 = 37 - 6 - 1$, that $184 = 228 - 37 - 6 - 1$, and so on. Thus we can imitate this structure to construct a subsequence of indices along which the Sudler product at $\beta$ shows the desired growth behaviour.\\
 \begin{proof}[Proof of Lemma \ref{lemma_b6}]
 	Let $\beta =[6,6,\ldots]= \sqrt{10}-3$, and let $(q_n)_{n=0}^{\infty}$ be the sequence of denominators of convergents to $\beta$. Set $N_k = q_k + q_{k-1} + \dots + q_1+q_0$ for $k \geq 1$. Furthermore, we set $M_j^{(k)} = q_k + q_{k-1} + \dots + q_{j+2} + q_{j+1}$, for $j=0, \dots, k-1$, and $M_k^{(k)}=0$. Then  
 	\begin{eqnarray}
 	P_{N_k} (\beta) & = & \prod_{r=1}^{N_k} 2 |\sin \pi r \beta | \, =\, \prod_{j=1}^{k} \prod_{r=M_j^{(k)}+1}^{M_j^{(k)} + q_j}\!\!\!2 |\sin \pi r \beta | \nonumber\\
 	& = & \prod_{j=1}^{k} \prod_{r=1}^{q_j} 2 |\sin \pi (r \beta + M_j^{(k)} \beta) | \nonumber\\
 	& = & \prod_{j=1}^{k} \prod_{r=1}^{q_j} 2 \left|\sin  \pi \left(r \beta + (-1)^{j} \frac{\ve_j^{(k)}}{q_j} \right)   \right| \nonumber\\
 	& = & \prod_{j=1}^{k} P_{q_j}(\beta,\ve_j^{(k)}), \label{p-fact}
 	\end{eqnarray}
 	where in view of \eqref{b1}, \eqref{b2} $\ve_j^{(k)}$ is defined so that
 	\begin{eqnarray*}
 		\frac{(-1)^{j} \ve_j^{(k)}}{q_j} & = & M_j^{(k)} \beta - q_{k-1} - q_{k_2} - \dots - q_{j+1} - q_j \\
 		& = & (q_k + q_{k-1} + \dots + q_{j+2} + q_{j+1}) \beta - q_{k-1} - q_{k_2} - \dots - q_{j+1} - q_j \\  [1ex]
 		& = & (-1)^{k} \beta^{k+1} + (-1)^{k-1} \beta^{k} + \dots + (-1)^{j+2} \beta^{j+3} + (-1)^{j+1} \beta^{j+2}.
 	\end{eqnarray*}
 	Accordingly
 	$$ \frac{\ve_j^{(k)}}{q_j \beta^{j+1}} =(-1)^{k-j} \beta^{k-j} + (-1)^{k-j-1} \beta^{k-j-1} + \dots + (-1)^{2} \beta^{2} + (-1)^{1} \beta^{1}.  $$
 	Since $q_1\beta^2 \leq q_j \beta^{j+1} \leq q_2\be^3$ for all $j=1,2,\ldots$  and
 	$$ -\frac{\be}{\sqrt{40}}\, \leq \,
 	\sum_{r=1}^n \frac{  (-\be)^{r} }{\sqrt{40}}\, \leq \,   -\frac{\be}{\sqrt{40}} + \frac{\be^2}{\sqrt{40}} , \quad n=1,2,\ldots $$
 	we deduce that  $\ve_j^{(k)} \in [-0.0257,-0.02]$ for all $k\geq 1$ and $ j< k$. Furthermore we can verify that 
 	
 	$G_{\be}(-0.257)  < G_\beta(-0.020) \approx 0.949$, up to an error of $\pm 0.01$ -- again this can be formally proved with an appropriate estimate for the approximation errors in \eqref{gformula}. \\
 	
 	\noindent\textit{Claim: } We have $G_\beta(\ve) \leq 0.96$ for any $\ve \in [-0.0257,-0.02]$.\\
 	\textit{Proof of Claim:} We know that $\log G_{\be}$ is strictly concave in the interval $[-1/\sqrt{40},(6+\be)/\sqrt{40} ]$ that is formed by two consecutive roots of $G_{\be}$. Consequently, $\log G_{\be}$ --and hence also $G_{\be}$-- is increasing in some interval $[-1/\sqrt{40}, s_0]$ and decreasing in $[s_0, (6+\be)/\sqrt{40}]$. 
 	We saw that $G_{\be}(-0.025)< G_{\be}(-0.02)< 0.96$ and also we can show that $G_{\be}(0)> 1.05 $, so $s_0>0$ and $G_\be$ is increasing in $[-1/\sqrt{40}, 0]$. The Claim now follows. \\
 	
 	\noindent Since the functions $P_{q_j}$ converge uniformly to $G_{\beta},$ we may bound all factors $P_{q_j}(\beta,\ve_j^{(k)})$ in \eqref{p-fact} from above by $0.97$ whenever $j \geq j_0$  for some appropriate $j_0$ (independent of $k$). Furthermore we can show that the remaining initial factors are bounded from below arguing as in the proof of Theorem \ref{th3}. This proves that $P_{N_k} (\beta) \ll 0.97^{k - j_0}$ and thus $P_{N_k}(\beta) \to 0$ as $k \to \infty$. Consequently, $\liminf\limits_{N \to \infty} P_N(\beta)=0$.	\\
 	
 	Proving that $\limsup \limits_{N\to\infty}\dfrac{P_N(\be)}{N} = \infty$ can be done in a perfectly analogous way. We define $N_k = q_{k+1} - q_k - q_{k-1} - \ldots - q_1$, and use the reflection principle \eqref{reflectionprinciple} adjusted to the case of the irrational $\be$. Then we are led to estimating products similar to the ones above, except that there is a additional perturbation $- q_{k+1} \beta$. However, the influence of this additional perturbation is very small (just one additional term in a geometric series -- cf.\ the remarks after the proof of Theorem \ref{th3}), so we can use exactly the same estimates as above in order to obtain the desired conclusion.
 \end{proof}

\noindent It remains to deal with the case when $b \in \{2,3,4,5\}$. 

\begin{lemma} \label{lemma_b2-5}
Let $\be$ be defined as in Theorem \ref{lineargrowththeorem}, for $b \in \{2,3,4,5\}$. Then 
$$\liminf \limits_{N \to \infty} P_N(\be) > 0 \qquad \text{and} \qquad \limsup \limits_{N\to\infty}\dfrac{P_N(\be)}{N} < \infty. $$
\end{lemma}

In principle we could settle all cases $b \in \{2,3,4,5\}$ simultaneously, keeping track of all possible structures of the perturbations $\ve$ coming from the Ostrowski expansions of the corresponding integers with respect to $\beta$. However, for the reader's convenience we decided to give a proof only for the particular case $b=5$, which is the most delicate one. The other cases $b \in \{2,3,4\}$ can be treated in a perfectly analogous way.\\

So we fix $b=5$, which means that $\beta = \frac{\sqrt{29}-5}{2}$. Before we start with the proof, let us give a heuristic description of what will happen. The first few denominators of convergents to $\beta$ are given by $1,5,26,135,701,\dots$. We have $G_\beta(0) = C_5 > 1$ by Corollary \ref{cor: when M-V const greater one}, and (as we will show) we have $G_\beta(\ve)>1$ for all possible \emph{positive} perturbations $\ve$ that could come from the Ostrowski expansion. Thus $\liminf P_N(\beta)=0$ could only happen as the effect of \emph{negative} perturbations $\ve$ for which $G_\beta(\ve)<1$. Since the differences $\beta - p_n/q_n$ have alternating signs, we know which structure in the Ostrowski expansion can give negative perturbations; essentially, there are the ones coming from an \emph{odd} difference in the index of the convergent denominator. Indeed, consider for example the case when $N = 31 = 26+5$. According to the Ostrowski expansion, we split the product $P_N(\beta)$ in the form 
\begin{eqnarray*}
P_N(\beta) & = & \prod_{r=1}^N 2 |\sin \pi r \beta | \\
& = & P_{26}(\beta) \cdot \prod_{r=1}^5 2 |\sin \pi (26 + r) \beta| \\
& = & P_{26}(\beta) \cdot \prod_{r=1}^5 2 |\sin \pi (\underbrace{26 \beta}_{\approx 5.007} + r \beta)| \\
& = & P_{26}(\beta) \cdot P_5 (\beta, \ve),
\end{eqnarray*}
where $\ve$ is the deviation of $26 \beta \cdot 5$ from the nearest integer (which is 25), and thus $\ve \approx -0.035$. Note that $\ve$ has negative sign. This essentially is a consequence of the fact that $26=q_3$ and $5 = q_2$, and that the difference of the indices is $3-2$, which is \emph{odd}.\\

To give a few other examples, when $N = 135 + 5$ where $135 = q_4$ and $5 = q_2$, then we get $P_{140}(\beta) = P_{135}(\beta) \cdot P_5 (\beta, \ve)$ with $\ve \approx 0.0069$, which is positive (and thus is good, since it implies $G_\beta(\ve)>1$). When we have $N = 701 + 5$ where $701=q_5$ and $5 = q_2$ then we get $P_{706}(\beta) = P_{701}(\beta) \cdot P_5 (\beta, \ve)$ with $\ve \approx -0.0013$, which again is negative (since $5-2=3$ is \emph{odd}).\\

So we saw that negative perturbations can only come from \emph{odd} differences of indices in the Ostrowski expansion. However, it turns out that the perturbation above, which was roughly $-0.035$, still does not cause us problems, since there $G_\beta$ still exceeds one (we have $G_\beta(-0.035) \approx 1.02$).\\

To reach a region where indeed $G_\beta<1$, we need the Ostrowski expansion of $N$ to have a specific structure. Consider now the number $N = 83 = 26 + 26 + 26 + 5$. According to the Ostrowski expansion, we decompose the product $P_N(\beta)$ into 
\begin{eqnarray*}
P_N(\beta) & = & \left(\prod_{r=1}^{26} 2 |\sin  \pi r \beta|\right) \cdot \left(\prod_{r=1}^{26} 2 |\sin \pi (26 + r) \beta|\right) \cdot \\
& & \cdot \left(\prod_{r=1}^{26} 2 |\sin\pi (26 + 26 + r) \beta|\right) \cdot \left(\prod_{r=1}^5 2 |\sin\pi (26 + 26 + 26 + r) \beta|\right).
\end{eqnarray*}
The last product is $P_5(\beta,\ve)$ with a perturbation $\ve$ coming from the difference between $3 \cdot 26 \beta \cdot 5$ and the nearest integer (which is 75), which gives the large negative perturbation $\ve \approx - 0.107$. We have $G_\beta(-0.107) \approx 0.54$, which is significantly smaller than 1, so in the decomposition of $P_N$ there is a factor which is much smaller than 1. However, note that we could only reach such a large negative perturbation by constructing $N$ in a way such that in the Ostrowski representation the same number (in our case 26) occurs more than once. Thus in the decomposition of $P_N$ we also see the products 
$$ \prod_{r=1}^{26} 2 |\sin\pi (26 + r) \beta| = P_{26}(\beta,\ve),   $$
with $\ve \approx 0.19$, and 
$$ \prod_{r=1}^{26} 2 |\sin\pi (26 + 26 + r)\beta| = P_{26}(\beta,\ve) $$
with $\ve \approx 0.37$, which means that we see two perturbed Sudler products $P_{26}$ with large \emph{positive} perturbations. This gives us two additional large factors of size roughly $G_\beta(0.19) \approx 2.27$ and $G_\beta(0.37) \approx 2.67$ in the decomposition of $P_N(\beta)$, which actually even overcompensate the influence of the single small factor of size roughly $0.54$.\\

\begin{figure}%
\centering
\subfigure[][]{%
\includegraphics[scale=0.5]{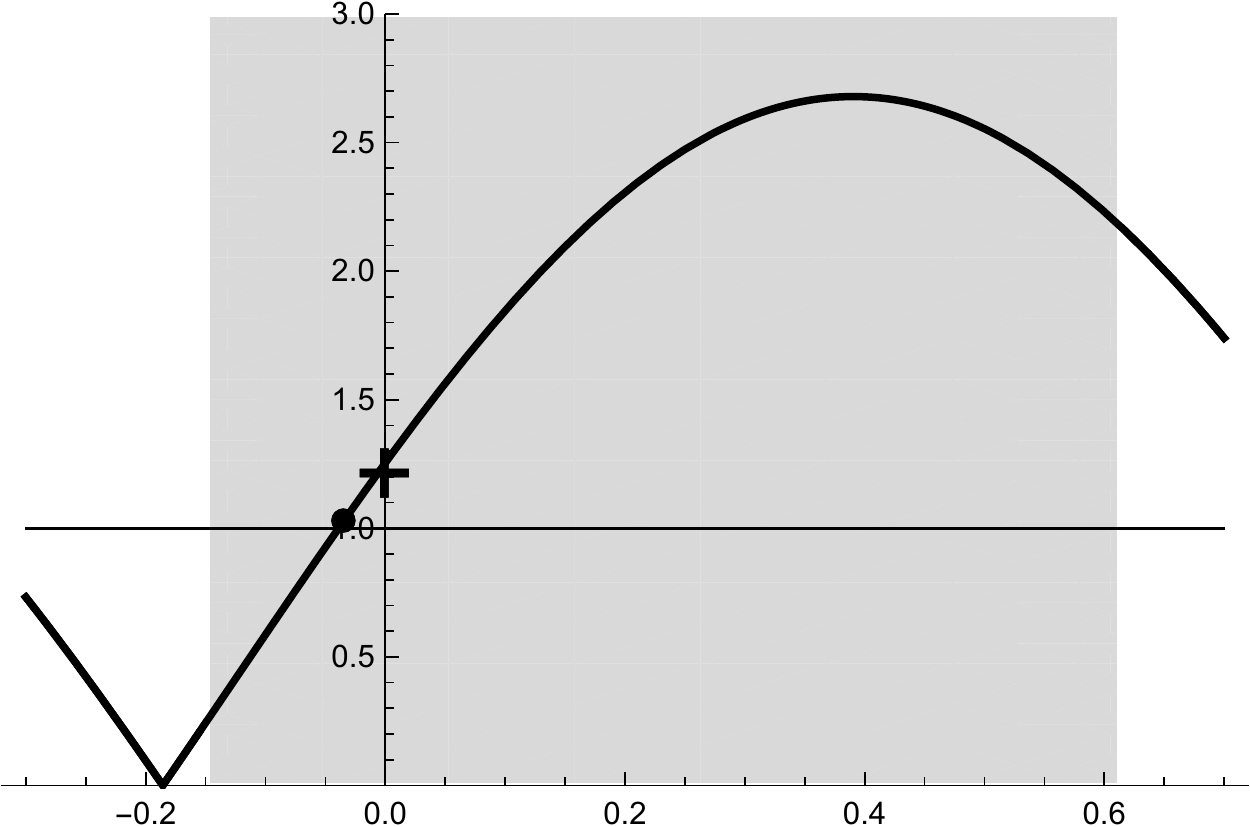}}%
\hspace{8pt}%
\subfigure[][]{%
\includegraphics[scale=0.5]{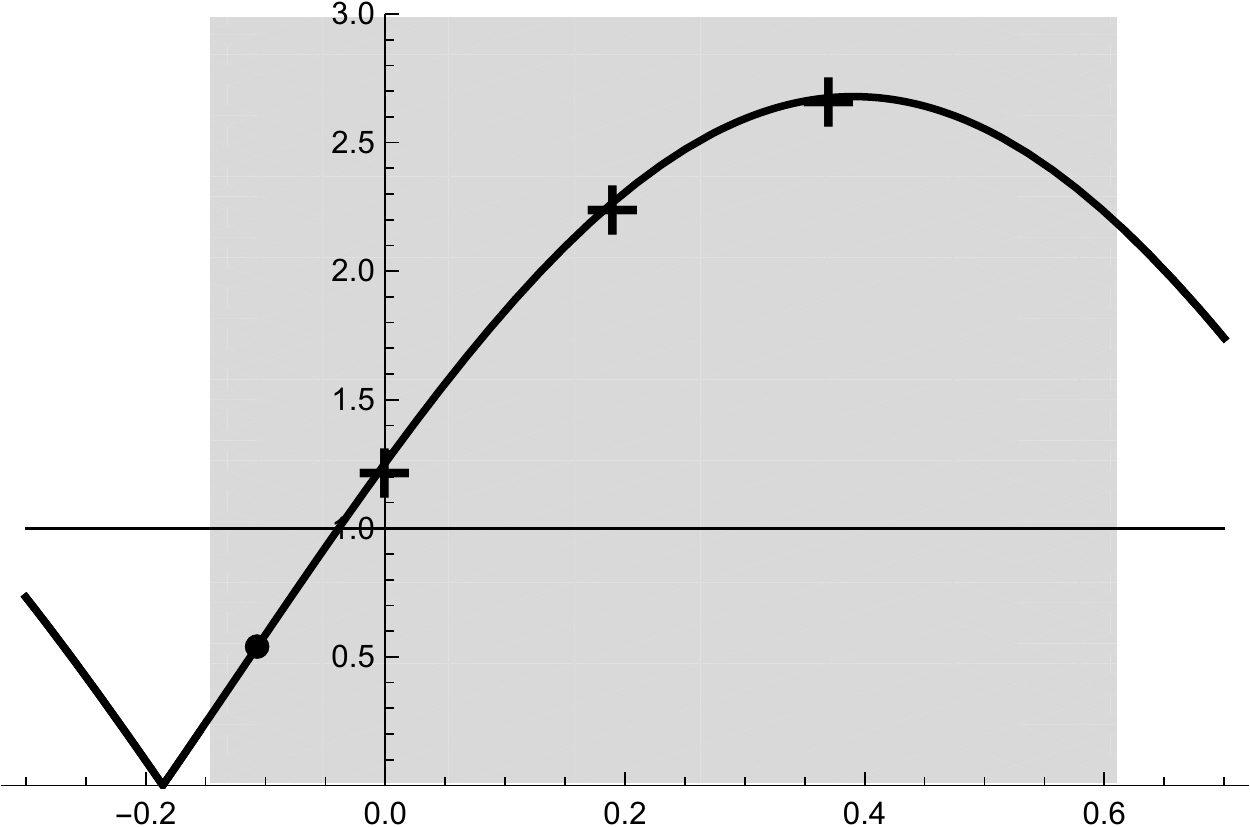}} \\
\caption{Two plots of $G_\beta(\ve)$ for $b=5$ in the range $-0.3 \leq \ve \leq 0.7$. The shaded region, ranging from ca.\ -0.15 to 0.61, is the region of possible perturbations $\ve$ which can arise from the Ostrowski representation used in the decomposition of $P_N(\be)$. Note that contributions smaller than 1 can only come from negative perturbations $\ve$. In the left picture, we see the case $N=31=26+5$. We have the contribution of the (unperturbed) product $P_{26}(\beta) = P_{26}(\beta,0)$, depicted by a cross symbol on the y-axis, and a contribution of roughly $P_5(\beta,-0.035) \approx G_\beta(-0.035)$, depicted by a dot symbol on the left of the y-axis. Note that $G_\beta(0) = C_5 > 1$, and that the negative perturbation in $P_5(\beta,-0.035)$ is so small that it still gives a factor which exceeds 1. The picture on the right shows the case $N = 3 \cdot 26 + 5$. Unlike in the golden ratio case, where we had $G_\phi(\ve)>1$ for all occurring perturbations $\ve$, there now is indeed one sub-product which leads to a perturbation value $\ve \approx -0.10$ such that $G_\be (\ve)$ is smaller than one (depicted by a dot). However, such a small value of $G_\be$ could only be reached by the way how the Ostrowski expansion of $N$ was constructed, with $N = 26+26+26+5$ containing the summand 26 three times. The whole product $P_N(\be)$ decomposes into additional products ranging from $1$ to $26$, from $26+1$ to $2\cdot 26$, and from $2 \cdot 26+1$ to $3 \cdot 26$, respectively. These correspond to values of $G_\be$ at perturbations ca.\ $0.00$, $0.19$ and $0.37$, respectively (depicted by three cross symbols). Since $G_\be$ is significantly greater than 1 at two of these three positions, the total contribution of these 3 large factors overcompensates the contribution of one small factor at the negative perturbation. This is the general principle that we will use in the proof of the theorem: critically large negative perturbations can only come from large digits in the Ostrowski expansion, but large digits always necessarily also generate positive perturbations.}%
\label{fig:ex3_a}%
\end{figure}

\begin{figure}[h]
        \includegraphics[scale=0.6]{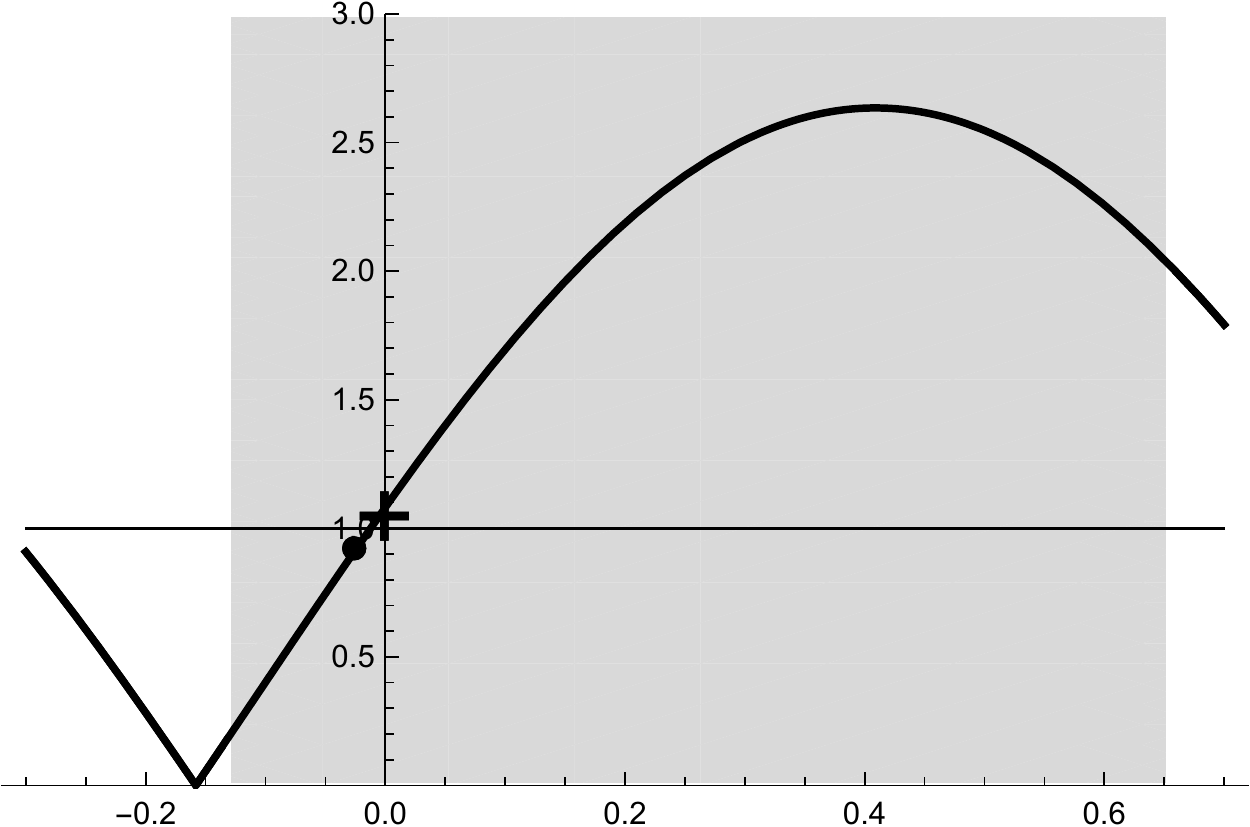}
        \caption{An analogue of the picture on the left-hand side of Figure 4, but for the case $b=6$ instead of $b=5$. Here $N=43=37+6$, which again is the sum of two consecutive convergent denominators. The product $P_{37}(\beta)$ is roughly 1.08. The perturbed product is roughly $P_6(\beta,-0.026) \approx G_\beta(-0.026) \approx 0.92$, which in contrast to the case $b=5$ now is smaller than 1. Choosing a number $N$ which is the sum of more convergent denominators to $\beta$, such as $N=1676=1405+228+37+6$, would give more contributions near the dot in the picture, all of them producing factors which are smaller than 1. This is the principle behind the proof of Lemma \ref{lemma>1}.}
    \end{figure}

This outlines the basic strategy for the proof of the theorem. We will make a case distinction, depending on whether a ``digit'' in the Ostrowski expansion of $N$ exceeds one or not. Whenever a digit is one or zero, it cannot lead to a large negative perturbation for the subsequent partial product, and everything is okay. If a digit exceeds one, then this gives large additional factors in our product decomposition, which overcompensate the potential small factor coming from the subsequent partial product which may have a large negative shift.\\

\begin{proof}[Proof of Lemma \ref{lemma_b2-5}]
For the arbitrary $N\geq 1$ let $n=n(N)\geq 1$ be such that $q_{n}\leq N+1 < q_{n+1}$. We can write $N$ in its Ostrowski representation in the form
$$ N = 	c_{n+1} q_n + c_n q_{n-1} + \dots + c_2 q_1 + c_1 q_0, $$
where we have $0 \leq c_1 < 5$ and 
$$
0 \leq c_i \leq 5 \quad \text{for all $i$} \qquad \text{and } \qquad \text{$c_i=0$ \, whenever \,  $c_{i+1}=5$.}
$$
For $0 \leq i \leq n$ and $0 \leq a_{i} < c_{n+1-i}$ we set 
$$
M_{i,a_i}  = c_{n+1} q_n + c_n q_{n-1} + \dots + c_{n+2-i} q_{n+1-i} + a_{i} q_{n-i},
$$
so that we can split the whole product $P_N(\beta)$ in the form
\begin{eqnarray*}
	P_N(\be) & = & \prod_{i=0}^n ~\prod_{a_i=0}^{c_{n+1-i} - 1}~ \prod_{r=M_{i,a_{i}}+1}^{M_{i,a_{i}}+q_{n-i}} 2 |\sin \pi r \beta| \\
	& = & \prod_{i=0}^n ~\prod_{a_i=0}^{c_{n+1-i} - 1}~ \prod_{r=1}^{q_{n-i}} 2 |\sin\pi (M_{i,a_{i}}+r) \beta|.
\end{eqnarray*}
To exploit the phenomenon addressed above, we will now split the product in such a way that we combine the contribution of ``large'' digits of $q_i$ with the contribution of the first (zero) digit of $q_{i-1}$. In formulas, we re-organize the product above in the form
\begin{eqnarray}
P_N(\be) & = & \prod_{i=0}^n ~\prod_{a_i=0}^{c_{n+1-i} - 1}~ \prod_{r=1}^{q_{n-i}} 2 |\sin\pi (M_{i,a_{i}}+r) \beta| \nonumber\\
& = & \left( \prod_{r=1}^{q_n} 2 |\sin\pi (M_{0,0}+r) \beta| \right) \cdot \label{prod_l} \\
& & \cdot \prod_{i=0}^n \left(\left(\prod_{a_i=1}^{c_{n+1-i} - 1}~ \prod_{r=1}^{q_{n-i}} 2 |\sin \pi (M_{i,a_{i}}+r) \beta | \right) \left(\prod_{r=1}^{q_{n-i-1}} 2 |\sin \pi (M_{i+1,0}+r) \beta | \right)\right)  \label{prod_l2}
\end{eqnarray} 
We have $M_{0,0}=0$, so for the product in line \eqref{prod_l} we simply have $\prod_{r=1}^{q_n} 2 |\sin \pi (M_{0,0}+r) \beta | \to C_b$ as $n\to\infty$. To emphasize the decomposition again, we have split the product in such a way that in \eqref{prod_l2} the contribution of the ``digits'' greater than one attached to some $q_j$ is combined with the contribution of the digit one attached to the next-smallest convergent denominator $q_{j-1}$.\\

Note that the product over $a_i$ in \eqref{prod_l2} is empty when $c_{n+1-i} \in \{0,1\}$. Similarly the final product in \eqref{prod_l2} can be empty, which happens when $c_{n-i} = 0$. We remind the reader that products over empty index sets are understood to equal 1.\\

So let us assume that we have chosen some value of $i$ in \eqref{prod_l2}. We wish to show that there is a $i_0$ such that
\begin{equation} \label{prod_ex}
\left(\prod_{a_i=1}^{c_{n+1-i} - 1}~ \prod_{r=1}^{q_{n-i}} 2 |\sin \pi (M_{i,a_{i}}+r) \beta | \right) \left(\prod_{r=1}^{q_{n-i-1}} 2 |\sin \pi (M_{i+1,0}+r) \beta | \right)
\end{equation}
is large enough whenever  $n-i \geq i_0$. Ideally we would wish that all factors in \eqref{prod_ex} are $>1$, since that would directly imply our desired result. It will turn out that it is not true that all factors in \eqref{prod_ex} exceed 1. While sometimes it may happen that such a factor is a bit smaller than 1, this will be compensated by other factors in \eqref{prod_ex} which exceed 1.\\

We distinguish the following cases depending on the values of $c_{n+1-i}, c_{n-i}$ and also, at times, depending 
on the values of $c_{n+2-i}$, and $c_{n+3-i}$.\\

\noindent $ \bullet $ Case 1: $c_{n-i} \neq 0$ and $c_{n+1-i}=0$.\\
  In this case the product $\prod\limits_{a_i=1}^{c_{n+1-i} - 1}\prod\limits_{r=1}^{q_{n-i}} 2 |\sin \pi (M_{i,a_{i}}+r) \beta | $ in \eqref{prod_ex} is empty, and we will \vspace{-2mm} \\
 show that
 \begin{equation} \label{to_show}
 \prod_{r=1}^{q_{n-i-1}} 2 |\sin \pi (M_{i+1,0}+r) \beta | > 1.
 \end{equation}
 Recall that by definition 
 $$
 M_{i+1,0}  = c_{n+1} q_n + c_n q_{n-1} + \dots + c_{n+1-i+1} q_{n-i+1} + c_{n+1-i} q_{n-i} + 0 \cdot q_{n-i-1},
 $$
 where the last term is zero because $a_{i+1}=0$. Since we assumed that $c_{n+1-i}=0$ the penultimate term also vanishes. Hence
 $$  M_{i+1,0}  = c_{n+1} q_n + c_n q_{n-1} + \dots + c_{n+2-i} q_{n-i+1}.  $$
 We write 
 \begin{eqnarray*}
 \prod_{r=1}^{q_{n-i-1}} 2 |\sin \pi (M_{i+1,0}+r) \beta |  & = & \prod_{r=1}^{q_{n-i-1}} 2 \left|\sin  \pi \left(r \beta + \frac{(-1)^{n-i-1} \ve}{q_{n-i-1}} \right) \right| =  P_{q_{n-i-1}} (\beta, \ve)  ,
 \end{eqnarray*}
 where  by \eqref{b1} and \eqref{b2} $\varepsilon$ can be chosen such that 
\begin{eqnarray*}
 \frac{(-1)^{n-i-1} \ve}{q_{n-i-1}} &=&  M_{i+1,0} \beta - (c_{n+1} q_{n-1} + \dots + c_{n+1-i+1} q_{n-i} ) \\
 &=& - c_{n+1} (-\beta)^{n+1} - c_n (-\beta)^{n} - \dots - c_{n+2-i} (- \beta)^{n+2-i}.
\end{eqnarray*}
Thus we have
\begin{eqnarray} \label{final_t}
\frac{\ve}{q_{n-i-1} \beta^{n-i}} & = & c_{n+1} (-1)^{i+1} \beta^{i+1} + c_n (-1)^{i} \beta^{i} + \dots + c_{n+2-i} (-1)^2 \beta^{2}.
\end{eqnarray}
 In the sum on the right of \eqref{final_t} there are both negative and positive contributions. Note that the last (largest) term is positive. Since the coefficients $c_{n+1}, ~c_n, \dots, c_{n+2-i}$ are all bounded above by 5, the total contribution of the positive terms is at most 
$$
\sum_{j=1}^\infty 5 \beta^{2j} = \beta,
$$
and the total contribution of the negative terms is at most
$$
- \sum_{j=1}^\infty 5 \beta^{2j+1} = -\beta^2.
$$
By \eqref{b2} we have $q_{n-i-1} \beta^{n-i} \to 1/\sqrt{29}$. Note that $\beta/\sqrt{29} \approx 0.036 < 0.04$ and $- \beta^2 / \sqrt{29} \approx -0.007 > -0.01$. So we have that
$$
\ve \in [-0.01,0.04],
$$
provided that $i$ is sufficiently large. We have $G_\beta(0.04) \approx 1.50$, and in particular we can formally prove (again using formula \eqref{pinftyformula} for $G_\beta$ and estimates for approximation errors for the infinite product) that $G_\beta(0.04) > 1.1$ . Similarly, we have $G_\beta(-0.01) \approx 1.19$, and we can formally prove that $G_\beta(-0.01) > 1.1$ . Thus by Proposition \ref{gbetaproperties}, the log-concavity of $G_\beta$ implies that $G_\beta(\ve) > 1.1$ throughout the whole range $[-0.01,0.04]$ of possible perturbations $\ve$. By the uniform convergence in Theorem \ref{th6} this means that $P_{q_{n-i-1}} (\beta, \ve)> 1.05$ whenever $i$ is sufficiently large. Thus we have established \eqref{to_show} for all $i \geq i_0$.  \par The  upshot is,  all  factors appearing in \eqref{prod_ex} which belong to Case $1$ are $>1$, except for finitely many of them;  the  overall contribution of the Case $1$ factors to the product in \eqref{prod_ex} can be bounded below by an absolute constant as in the proof of Theorem \ref{th3}.\\

\noindent $\bullet$ Case $2$: $c_{n-i} \neq 0$ and $c_{n+1-i}=1$.\\
In this case again the product  $\prod\limits_{a_i=1}^{c_{n+1-i} - 1}\prod\limits_{r=1}^{q_{n-i}} 2 |\sin \pi (M_{i,a_{i}}+r) \beta | $  in \eqref{prod_ex} is empty, and we \vspace{-2mm} \\
have to control the size of 
 \begin{equation} \label{to_show_2}
 \prod_{r=1}^{q_{n-i-1}} 2 |\sin \pi (M_{i+1,0}+r) \beta |.
 \end{equation}
Now we have
 $$
 M_{i+1,0}  = c_{n+1} q_n + c_n q_{n-1} + \dots + c_{n+1-i+1} q_{n-i+1} + \underbrace{c_{n+1-i}}_{=1} q_{n-i}.
 $$
 Similar to Case 1, we can write the product in \eqref{to_show_2} in the form
 $$
 \prod_{r=1}^{q_{n-i-1}} 2 |\sin \pi (M_{i+1,0}+r) \beta| = P_{q_{n-i-1}} (\beta, \ve),
 $$
 where instead of \eqref{final_t} we now obtain 
 \begin{eqnarray} 
 \frac{\ve}{q_{n-i-1} \beta^{n-i}}  \label{final_t_2} &=& c_{n+1} (-1)^{i+1} \beta^{i+1} + c_n (-1)^{i} \beta^{i} + \dots + c_{n+1-i+1} (-1)^2 \beta^{2} + (-1)^1 \beta^1. \nonumber
\end{eqnarray}
Here the last term ``$+(-1)^1 \beta^1$'' comes from the assumption that $c_{n+1-i}=1$. Again, we find the maximal positive and negative contributions to determine the range of all possible perturbations $\ve$. We have
$$ \frac{\ve}{q_{n-i-1} \beta^{n-i}} \leq - \beta +   \sum_{j=2}^\infty 5 \beta^{2j} < 0  $$
and 
\begin{eqnarray} 
\frac{\ve}{q_{n-i-1} \beta^{n-i}} & \geq & - \beta + c_{n+2-i} \beta^2 - c_{n+3-i} \beta^3 -   \sum_{j=2}^\infty 5 \beta^{2j+1} \nonumber\\
& = & -\beta + c_{n+2-i} \beta^2 - c_{n+3-i} \beta^3 - \beta^4.\label{detail}
\end{eqnarray}

We know that $G_\beta(0) = C_5 \approx 1.25$, and we can formally verify that $G_\beta(0) > 1.1$, so $P_{q_{n-i-1}} (\beta, \ve) > 1.1 $ for all $i\geq i_0$.  However, the situation is more delicate regarding the lower bound of possible perturbations: suppose that in \eqref{detail} we ignore the influence of the digits $c_{n+2-i}$ and $c_{n+3-i}$, and just use the estimates $c_{n+2-i} \geq 0$ and $c_{n+3-i} \leq 5$. This would lead to $\ve \geq (- \beta - \beta^2)/\sqrt{29}$. Now $(-\beta - \beta^2)/\sqrt{29} \approx -0.04$, but $G_\beta(-0.04<1) .$  So we need to provide a sharper estimate, and we distinguish further sub-cases based on the values of  $c_{n+2-i}$ and $c_{n+3-i}$.\\

\begin{itemize}
\item[--] Case 2a: $c_{n+3-i} \leq 2$ or $c_{n+2-i} \neq 0$.\\
In this case from \eqref{detail} we can deduce that 
	$$ \frac{\ve}{q_{n-i-1} \beta^{n-i}} \,\geq \, \min \{- \beta - 2 \beta^3 - \beta^4,\,  - \beta + \beta^2 - 5 \beta^3 - \beta^4\}.
	$$ Since $q_{n-i-1} \beta^{n-i} \to 1/\sqrt{29},$ this implies that $\ve > -0.0387$ for sufficiently large $n$. We have $G_\beta(-0.0387) \approx 1.002$, and we can formally verify that $G_\beta(-0.0387) > 1.001$. We noted above that $G_\beta(0)>1.1$. Thus we have $G_\beta(\ve) > 1.001$ for all possible perturbations in Case 2a, and thus $P_{q_{n-i-1}} (\beta, \ve) > 1.0001$ for all sufficiently large $i$. Thus all factors in Case 2a exceed 1, except for finitely many factors coming from indices $i < i_0$ that can be treated similarly as in the proof of Theorem \ref{th3}. Thus the overall contribution of the Case 2a factors is bounded below by a positive absolute constant.\\
	
	\item[--] Case 2b: $c_{n+3-i} \geq 3$ and $c_{n+2-i} = 0$.\\
	
	In this case from \eqref{detail} we have $\ve/(q_{n-i-1} \beta^{n-i})  \geq - \beta - 5 \beta^3 - \beta^4$, and thus by $q_{n-i-1} \beta^{n-i} \to 1/\sqrt{29}$ we have $\ve \geq - 0.043$. We have $G_\beta(-0.043) \approx 0.973$, and we can formally prove that $G(-0.043) > 0.97$       . Note that in Case 2b we can thus \emph{not} guarantee that we have a factor which is 1 or less. Instead we can only deduce that $P_{q_{n-i-1}} (\beta, \ve) > 0.96$, say, for all sufficiently large $i$. Note that Case 2b only occurs when the Ostrowski expansion has a ``digit'' of at least 3, followed by a zero ``digit''. Thus the joint overall contribution of the Case 2b factors is not less that an absolute constant multiplied with 
	\begin{equation} \label{small_c}
	0.96^{A}, \qquad \textrm{where $A = \# \{2 \leq i \leq n:~c_i \geq 3,~ \text{and} ~ c_{i-1} \neq 0 \}$}
	\end{equation}
	(as always, the absolute constant comes from the constribution of finitely many indices for which $i<i_0$). We will need to show that the ``small'' contribution of \eqref{small_c} to the product \eqref{prod_ex} is compensated by an overshoot in the contribution of Case 6.\\
\end{itemize}

\noindent $\bullet$ Case 3:  $c_{n-i} \neq 0$ and $c_{n+1-i}=2$.\\

\noindent In this case the product  $\prod\limits_{a_i=1}^{c_{n+1-i} - 1}\prod\limits_{r=1}^{q_{n-i}} 2 |\sin \pi (M_{i,a_{i}}+r) \beta | $ in \eqref{prod_ex} is not empty, and we \vspace{-2mm} \\
 need  to show that 
\begin{equation} \label{to_show_3}
\left(\prod_{r=1}^{q_{n-i}} 2 |\sin \pi (M_{i,1}+r) \beta| \right) \left(\prod_{r=1}^{q_{n-i-1}} 2 |\sin \pi (M_{i+1,0}+r) \beta |\right)  > 1.
\end{equation}
We analyse the two products separately, and show that their product exceeds 1. It is crucial to combine these products, since the second product in \eqref{to_show_3} alone does \emph{not} necessarily exceed 1, and we need the first product for compensation. We have 
$$ M_{i,1} =  c_{n+1} q_n + c_n q_{n-1} + \dots + c_{n+2-i} q_{n-i+1} + q_{n-i} $$
and 
$$ M_{i+1,0} = c_{n+1} q_n + c_n q_{n-1} + \dots + c_{n+2-i} q_{n-i+1} + 2 q_{n-i} . $$
The product \eqref{to_show_3} equals 
\begin{equation} \label{equals_what}
P_{q_{n-i}} (\beta, \ve_1) P_{q_{n-i-1}} (\beta, \ve_2),
\end{equation}
where 
 \begin{eqnarray*} 
 \frac{\ve_1}{q_{n-i} \beta^{n-i+1}}
 & = & c_{n+1} (-1)^{i} \beta^{i} + c_n (-1)^{i-1} \beta^{i-1} + \dots + c_{n+2-i} (-1)^1 \beta^{1} + \underbrace{1 (-1)^0 \beta^0}_{=1}  \nonumber
 \end{eqnarray*}
  and 
 \begin{eqnarray*} 
  \frac{\ve_2}{q_{n-i-1} \beta^{n-i}}  & = & c_{n+1} (-1)^{i+1} \beta^{i+1} + c_n (-1)^{i} \beta^{i} + \dots + c_{n+2-i} (-1)^2 \beta^{2} + 2 (-1)^1 \beta^1.
 \end{eqnarray*}
Note that $c_{n+1-i}=2$ implies that $c_{n+2-i} \leq 4$. Thus we have
$$
\frac{\ve_1}{q_{n-i} \beta^{n-i+1}} \leq 1 + 5 \sum_{j=1}^\infty \beta^{2j} = 1 + \beta,
$$
and 
$$
\frac{\ve_1}{q_{n-i} \beta^{n-i+1}} \geq 1 - 4 \beta - 5 \sum_{j=1}^\infty \beta^{2j+1} \geq 1 - 4 \beta - \beta^2. 
$$
Again using that $q_{n-i-1} \beta^{n-i} \to 1/\sqrt{29}$, this implies that $\ve_1 \in [0.03,0.23]$ for sufficiently large $i$. Similarly, we estimate the range for $\ve_2$ and obtain
$$
\frac{\ve_2}{q_{n-i-1} \beta^{n-i}} \leq -2 \beta + 5 \sum_{j=1}^\infty \beta^{2j} \leq - \beta,
$$
as well as
$$
\frac{\ve_2}{q_{n-i-1} \beta^{n-i}} \geq - 2 \beta - 5 \sum_{j=1}^\infty \beta^{2j+1} = -2 \beta - \beta^2.
$$
Thus we have $\ve_2 \in [-0.079,-0.036]$ for sufficiently large $i$. We can establish that $G_\beta(\ve) > 1.44$ throughout the possible range for $\ve_1$, and that $G_\beta(\ve) > 0.72$ throughout the possible range for $\ve_2$. Thus $P_{q_{n-i}} (\beta, \ve_1) > 1.43$ and $P_{q_{n-i-1}} (\beta, \ve_2) > 0.71$ whenever $i \geq i_0$, which implies that $P_{q_{n-i}} (\beta, \ve_1) P_{q_{n-i-1}} (\beta, \ve_2) > 1.43 \cdot 0.71 > 1.01$ whenever $i \geq i_0$. Thus the overall joint contribution of the Case 3 factors is bounded below by a positive constant (coming from the factors for which $i \geq i_0$ is not satisfied). \\

\noindent $\bullet$ Case 4:  $c_{n-i} \neq 0$ and $c_{n+1-i}=3$.\\

\noindent In this case we wish to obtain a lower bound for the product
\begin{equation} \label{to_show_4}
\left(\prod_{r=1}^{q_{n-i}} 2 |\sin \pi (M_{i,2}+r) \beta | \right)\!\left(\prod_{r=1}^{q_{n-i}} 2 |\sin\pi (M_{i,1}+r) \beta| \right)\!\left(\prod_{r=1}^{q_{n-i-1}}\!2|\sin\pi (M_{i+1,0}+r) \beta |\right)\!\!.
\end{equation}
This product is equal to
\begin{equation} \label{this_prod}
P_{q_{n-i}} (\beta, \ve_1) P_{q_{n-i}} (\beta, \ve_2) P_{q_{n-i-1}} (\beta, \ve_3),
\end{equation}
where  for $j=1,2$
\begin{eqnarray*}
  \frac{\ve_j}{q_{n-i} \beta^{n-i+1}} &=& c_{n+1} (-1)^{i} \beta^{i} + c_n (-1)^{i-1} \beta^{i-1} + \dots + c_{n+2-i} (-1)^1 \beta^{1} + j
\end{eqnarray*}
 and 
\begin{eqnarray*}
\frac{\ve_3}{q_{n-i-1} \beta^{n-i}} &=& c_{n+1}(-1)^{i+1}\beta^{i+1} + c_n(-1)^{i} \beta^{i} + \dots + c_{n+2-i} (-1)^2 \beta^{2} + 3(-1)^1 \beta^1.
\end{eqnarray*} 
Using a similar analysis as in Case 3, we obtain the restrictions $\ve_1 \in [0.03,0.23],~\ve_2 \in [0.21,0.42]$, and $\ve_3 \in [-0.115,-0.07]$. In these respective ranges the function $G_\beta$ is uniformly bounded below by the values $1.44$, $2.34$ and $0.48$. Note that $1.44 \cdot 2.34 \cdot 0.48 \approx 1.62$. Thus we have $P_{q_{n-i}} (\beta, \ve_1) P_{q_{n-i}} (\beta, \ve_2) P_{q_{n-i-1}} (\beta, \ve_3) > 1.61$ whenever $i \geq i_0$ for appropriate $i_0$. Consequently the joint overall contribution of the Case 4 factors is bounded below by an absolute constant.\\

\noindent $\bullet$ Case 5: $c_{n-i} \neq 0$ and $c_{n+1-i}=4$.\\

\noindent Now we need to control the product 
\begin{equation} \label{this_prod_2}
P_{q_{n-i}} (\beta, \ve_1) P_{q_{n-i}} (\beta, \ve_2) P_{q_{n-i}} (\beta, \ve_3) P_{q_{n-i-1}} (\beta, \ve_4),
\end{equation}
where 
\begin{eqnarray*}
  \frac{\ve_j}{q_{n-i} \beta^{n-i+1}}  
& = & c_{n+1} (-1)^{i} \beta^{i} + c_n (-1)^{i-1} \beta^{i-1} + \dots + c_{n+3-i} (-1)^2 \beta^{2} + c_{n+2-i } (-1)^1 \beta^{1} + j
\end{eqnarray*}
for $j= 1,2,3$, and where 
\begin{eqnarray*}
  \frac{\ve_4}{q_{n-i-1} \beta^{n-i}} &= & c_{n+1} (-1)^{i+1} \beta^{i+1} + c_n (-1)^{i} \beta^{i} + \dots + c_{n+2-i} (-1)^2 \beta^{2} + 4 (-1)^1 \beta^1.
\end{eqnarray*} 
This gives us the restrictions $\ve_1 \in [0.03,0.23],~\ve_2 \in [0.21,0.42], ~\ve_3 \in [0.39,0.61]$ and $\ve_4 \in [-0.151,-0.10]$. Throughout these ranges the function $G_\beta$ is uniformly bounded below by $1.44, ~2.34, ~2.18$ and $0.23$, respectively. We have $1.44 \cdot 2.34 \cdot 2.18 \cdot 0.23 \approx 1.69$. Thus the overall contribution of the Case 5 factors is bounded below by an absolute constant, since the finitely many factors corresponding to indices $n-i\leq i_0$ can be bounded from below arguing as in the proof of Theorem \ref{th3}.\\

\noindent $\bullet$ Case 6:  $c_{n-i} = 0$.\\

\noindent Note that this case includes the case $c_{n+1-i}=5$, since necessarily we always have $c_{n-i} = 0$ when $c_{n+1-i}=5$. We distinguish five subcases:\\

\begin{itemize}
	\item[--] Case 6a: $c_{n+1-i} \in \{0,1\}$. In this case all products in \eqref{prod_ex} are empty, and the value of an empty product is 1.\\
	
	\item[--] Case 6b: $c_{n+1-i} = 2$. We can estimate similar to Case 3 above, but since $c_{n-i} = 0$ we only have $P_{q_{n-i}} (\beta, \ve_1)$ instead of \eqref{equals_what}. As we showed in Case 3 above, we have $P_{q_{n-i}} (\beta, \ve_1) > 1.43$, except for finitely many indices. Thus the joint overall contribution of Case 6b factors is bounded below by an absolute constant.\\
	
	\item[--] Case 6c: $c_{n+1-i} = 3$. In the same way that Case 6b was reduced to Case 3, we can reduce this case to Case 4 from above. We have $P_{q_{n-i}} (\beta, \ve_1) P_{q_{n-i}} (\beta, \ve_2)$ instead of \eqref{this_prod}, and by the Case 4 analysis this can be bounded below by a constant below the product $1.44 \cdot 2.34$, such as $3$. Thus the joint overall contribution of Case 6c factors is bounded below by an absolute constant, multiplied with 
	$$
	3^{A^{(3)}}, \qquad \textrm{where $A^{(3)} = \# \{2 \leq i \leq n:~c_i = 3 ~ \text{and} ~ c_{i-1} = 0 \}$}.
	$$
	
	\item[--] Case 6d: $c_{n+1-i} = 4$. This can be reduced to Case 5 from above. We have $P_{q_{n-i}} (\beta, \ve_1) P_{q_{n-i}} (\beta, \ve_2) P_{q_{n-i}} (\beta, \ve_3)$ instead of \eqref{this_prod_2}, and according to the Case 5 analysis this is bounded below by any constant below $1.44 \cdot 2.34 \cdot 2.18$, such as $7$ (except for the contribution of finitely many indices). Thus the joint overall contribution of Case 6d factors is bounded below by an absolute constant, multiplied with
	$$
	7^{A^{(4)}}, \qquad \textrm{where $A^{(4)} = \# \{2 \leq i \leq n:~c_i = 4 ~ \text{and} ~ c_{i-1} = 0 \}$}.
	$$
	
	\item[--] Case 6e: $c_{n+1-i} = 5$. In this case we have to work a bit again, since we encounter a large \emph{positive} perturbation whose influence has to be controlled. In this case we have to control a product that can be written in the form
	\begin{equation} \label{product_in}
	P_{q_{n-i}} (\beta, \ve_1) P_{q_{n-i}} (\beta, \ve_2) P_{q_{n-i}} (\beta, \ve_3) P_{q_{n-i}} (\beta, \ve_4),
	\end{equation}
	where 
\begin{eqnarray*}
\frac{\ve_j}{q_{n-i} \beta^{n-i+1}} & = & c_{n+1} (-1)^{i} \beta^{i} +  \ldots + c_{n+3-i} (-1)^2 \beta^{2} + c_{n+2-i} (-1)^1 \beta^{1} + j,
\end{eqnarray*}
	for $j=1,2,3,4$. This gives us the restrictions 
	$$
	\ve_1 \in [0.03,0.23],~\ve_2 \in [0.21,0.42], ~\ve_3 \in [0.39,0.61],~\ve_4 \in [0.57,0.80].
	$$
 In these ranges $G_\beta$ is uniformly bounded below by the values $1.44,~2.34,~2.18$ and $1.12$, respectively. We have $1.44 \cdot 2.34 \cdot 2.18 \cdot 1.12 \approx 8.23$, so the product in \eqref{product_in} is always bounded below by $8$ (provided that $n-i \geq i_0$). So the total contribution of the Case 6e factors is bounded below by an absolute constant (coming from the terms with $n-i <i_0$ ; see also the proof of Theorem \ref{th3}), multiplied with
	$$ 	8^{A^{(5)}}, \qquad \textrm{where $A^{(5)} = \# \{2 \leq i \leq n:~c_i = 5 ~ \text{and} ~ c_{i-1} = 0 \}$}.
	$$ 
\end{itemize}

Finally, we collect the contribution of all cases. The overall contribution of each of the Case 1, 2a, 3, 4, 5, 6a, 6b and 6e factors to the product in \eqref{prod_ex} is bounded below by an absolute constant. The joint overall contribution of the Case 2b factors in bounded below by
$$
0.9^{A}, \qquad \textrm{where $A = \# \{2 \leq i \leq n:~c_i \geq 3 ~ \text{and} ~ c_{i-1} = 0 \}$},
$$
while the joint overall contribution of the Case 6c, Case 6d and Case 6e factors was bounded below by
$$
3^{A^{(3)}}, \qquad \textrm{where $A^{(3)} = \# \{2 \leq i \leq n:~c_i = 3 ~ \text{and} ~ c_{i-1} = 0 \}$},
$$
by 
$$
7^{A^{(4)}}, \qquad \textrm{where $A^{(4)} = \# \{2 \leq i \leq n:~c_i = 4 ~ \text{and} ~ c_{i-1} = 0 \}$},
$$
and by
$$
8^{A^{(5)}}, \qquad \textrm{where $A^{(5)} = \# \{2 \leq i \leq n:~c_i = 5 ~ \text{and} ~ c_{i-1} = 0 \}$},
$$
respectively (up to multiplication with an absolute constant). 
Now note that $A = A^{(3)} + A^{(4)} + A^{(5)}$. 
Thus the joint overall contribution of the Case 2b 
and the Case 6c, 6d, 6e 
factors together is bounded below by an absolute constant 
multiplied by the product 
$0.9^A \cdot 3^{A^{(3)}} \cdot 7^{A^{(4)}} 
\cdot 8^{A^{(5)}} \geq 0.9^A 
\cdot 3^{A^{(3)}+A^{(4)}+A^{(5)}} = 0.9^A \cdot 3^A > 1$.\\

Combining all these estimates proves that \eqref{prod_ex} 
is bounded below by an absolute constant. Consequently $P_N(\beta)$ is
bounded below by an absolute constant, as desired.\\

We can establish that $\limsup_{N \to \infty} P_N(\beta)/N < \infty$ in a completely analogous way, using the reflection principle \eqref{reflectionprinciple} for the irrational $\beta$ as in the proof of Theorem \ref{th3}. This ``reflection'' only generates an additional very small perturbation (coming from the reflection at the endpoint), but this does not actually change the range for the permissible perturbations $\ve$ in all our estimates above, since this additional small perturbation only appears as one further term within a geometric progression, and we have always used the estimated coming from summation of the whole infinite geometric progression. So the proof of $\limsup_{N \to \infty} P_N(\beta)/N < \infty$ can be carried out in exactly the same way as the proof of $\liminf\limits_{N \to \infty} P_N(\beta)>0$ above.
\end{proof} 

We conclude the manuscript with the corollary announced 
after Theorem \ref{lineargrowththeorem}. Loosely speaking, 
its purpose is to record that the class 
of quadratic irrationals that we investigated 
is not overly exceptional.
In this context, it would clearly be interesting 
to know the answer to the following problem:
\begin{quest}\label{question1}
For which $\beta\in [0,1]$ does the
associated Sudler product $P_N(\beta)$ grow at most linearly,
i.e. \eqref{folkloreconj} is satisfied?
\end{quest}
A first step towards a solution of this question would be to completely settle the case of Sudler products of numbers
whose continued fraction expansion is two-periodic:
\begin{quest}\label{question2}
For which numbers $\beta$ of the form $\beta=[a,b,a,b,a,b,\ldots]$, with $a \neq b$, is \eqref{folkloreconj} satisfied?
\end{quest}

To formulate our corollary, which is a first
partial result towards and answer of Question \ref{question1}, 
we let $H(\alpha)= \max_{0\leq i\leq d} \vert a_i\vert$ denote the (naive) height of an algebraic number $\alpha \in \mathbb{C}$ 
with minimal polynomial 
$$\mu_{\alpha}(X)
=\sum_{0\leq i\leq d} a_i X^i.
$$
Furthermore, it is well-known that there are 
$\asymp X^3$ many (real) quadratic irrationals 
of height at most $X$.
\begin{cor}\label{height corr}
There are $\gg X$ many quadratic 
irrationals 
$\beta\in \mathbb{R}$ of height at most $X$ with
$$\liminf \limits_{N \to \infty} P_N(\be) = 0, 
\quad \mathrm{and} \quad 
\limsup \limits_{N\to\infty}\dfrac{P_N(\be)}{N} = \infty.
$$
\end{cor}
\begin{proof}
Let $\beta$ be as in \eqref{eq: explicit form of beta}
where $b$ is an integer such that $b^2+4$ is not a perfect square.
We note that, for any integer $c$, the number
$\gamma(b,c) = \beta + c$ is a quadratic
irrationality whose minimal polynomial is seen,
by using \eqref{eq: explicit form of beta}, to be
$$\mu(X)=(2(X-c)+b)^2-(b^2+4)= 4X^2 + 2(b-2c)X+4c^2-4bc-4.
$$
Choosing $b,c \in (\sqrt{X}/100, \sqrt{X}/200)$ as above, 
the coefficients of the aforementioned
polynomial are $\leq X$. 
A well-known theorem of Besicovitch \cite{besicovitch},
about the $\mathbb{Q}$-linear
independence of square-roots, implies that all these $\gg X$
many numbers $\gamma(b,c)$ are pairwise distinct.
\end{proof}

\subsection*{Acknowledgements}
We thank Martin Widmer for bringing Green's comment 
\cite{green} to our attention,
which was one of the starting points for writing this paper. 
\\

\end{document}